\documentclass[11pt]{amsart}

\usepackage[dvips]{graphicx}
\usepackage{calc}
\usepackage{color}
\usepackage{amsmath}
\usepackage{amssymb}
\usepackage{amscd}
\usepackage{amsthm}
\usepackage{amsbsy}
\usepackage{delarray}
\usepackage{enumerate}
\usepackage[T1]{fontenc}
\usepackage{inputenc}
\usepackage{enumerate}
\usepackage{hyperref}

\begin{document}



\setlength{\parindent}{5mm}
\renewcommand{\leq}{\leqslant}
\renewcommand{\geq}{\geqslant}
\newcommand{\N}{\mathbb{N}}
\newcommand{\sph}{\mathbb{S}}
\newcommand{\Z}{\mathbb{Z}}
\newcommand{\R}{\mathbb{R}}
\newcommand{\C}{\mathbb{C}}
\newcommand{\F}{\mathbb{F}}
\newcommand{\g}{\mathfrak{g}}
\newcommand{\h}{\mathfrak{h}}
\newcommand{\K}{\mathbb{K}}
\newcommand{\RN}{\mathbb{R}^{2n}}
\newcommand{\ci}{c^{\infty}}
\newcommand{\derive}[2]{\frac{\partial{#1}}{\partial{#2}}}
\renewcommand{\S}{\mathbb{S}}
\renewcommand{\H}{\mathbb{H}}
\newcommand{\eps}{\varepsilon}
\newcommand{\chzrel}{c_{\mathrm{LR}}}

\theoremstyle{plain}
\newtheorem{theo}{Theorem}
\newtheorem{prop}[theo]{Proposition}
\newtheorem{lemma}[theo]{Lemma}
\newtheorem{definition}[theo]{Definition}
\newtheorem*{notation*}{Notation}
\newtheorem*{notations*}{Notations}
\newtheorem{corol}[theo]{Corollary}
\newtheorem{conj}[theo]{Conjecture}

\newenvironment{demo}[1][]{\addvspace{8mm} \emph{Proof #1.
    ~~}}{~~~$\Box$\bigskip}

\newlength{\espaceavantspecialthm}
\newlength{\espaceapresspecialthm}
\setlength{\espaceavantspecialthm}{\topsep} \setlength{\espaceapresspecialthm}{\topsep}

\newenvironment{example}[1][]{
\vskip \espaceavantspecialthm \noindent \textsc{Example
#1.} }%
{\vskip \espaceapresspecialthm}

\newenvironment{question}[1][]{
\vskip \espaceavantspecialthm \noindent \textsc{Question
#1.} }%
{\vskip \espaceapresspecialthm}

\newenvironment{remark}[1][]{\refstepcounter{theo} 
\vskip \espaceavantspecialthm \noindent \textsc{Remark~\thetheo
#1.} }%
{\vskip \espaceapresspecialthm}

\def\bb#1{\mathbb{#1}} \def\m#1{\mathcal{#1}}

\def\momeg{(M,\omega)}
\def\co{\colon\thinspace}
\def\Homeo{\mathrm{Homeo}}
\def\Hameo{\mathrm{Hameo}}
\def\Diffeo{\mathrm{Diffeo}}
\def\Symp{\mathrm{Symp}}
\def\Sympeo{\mathrm{Sympeo}}
\def\Id{\mathrm{Id}}
\newcommand{\norm}[1]{||#1||}
\def\Ham{\mathrm{Ham}}
\def\lagham#1{\mathcal{L}^\mathrm{Ham}({#1})}
\def\Hamtilde{\widetilde{\mathrm{Ham}}}
\def\cOlag#1{\mathrm{Sympeo}({#1})}
\def\Crit{\mathrm{Crit}}
\def\Spec{\mathrm{Spec}}
\def\osc{\mathrm{osc}}
\def\Cal{\mathrm{Cal}}

\title[Coisotropic $C^0$--rigidity]{Coisotropic rigidity and $C^0$--symplectic geometry}
\author{Vincent Humili\`ere, R\'emi Leclercq, Sobhan Seyfaddini}
\date{\today}

\address{VH: Institut de Math\'ematiques de Jussieu, Universit\'e Pierre et Marie Curie, 4 place Jussieu, 75005 Paris, France}
\email{vincent.humiliere@imj-prg.fr}

\address{RL: Universit\'e Paris-Sud, D\'epartement de Math\'ematiques, Bat. 425, 91405 Orsay Cedex, France}
\email{remi.leclercq@math.u-psud.fr}

\address{SS: D\'epartement de Math\'ematiques et Applications de l'\'Ecole Normale Sup\'erieure, 45 rue d'Ulm, F 75230 Paris cedex 05}
\email{sobhan.seyfaddini@ens.fr}


\maketitle

\begin{abstract}
We prove that symplectic homeomorphisms, in the sense of the celebrated Gromov--Eliashberg Theorem, preserve coisotropic submanifolds and their characteristic foliations.  This result generalizes the  Gromov--Eliashberg Theorem and demonstrates that previous rigidity results (on Lagrangians by Laudenbach--Sikorav, and on characteristics of hypersurfaces by Opshtein) are manifestations of a single rigidity phenomenon. 
To prove the above, we establish a $C^0$--dynamical property of coisotropic submanifolds which generalizes a foundational theorem in $C^0$--Hamiltonian dynamics: Uniqueness of generators for continuous analogs of Hamiltonian flows.  
\end{abstract}

\section{Introduction and main results}\label{sec:introduction}

 A submanifold $C$ of a symplectic manifold $(M, \omega)$ is called coisotropic if for all $p\in C$, $(T_pC)^\omega\subset T_pC$ where $(T_pC)^\omega$ denotes the symplectic orthogonal of $T_pC$. For instance, hypersurfaces and Lagrangians are coisotropic. A coisotropic submanifold carries a natural foliation $\m F$ which integrates the distribution $(TC)^\omega$; $\m F$ is called the characteristic foliation of $C$. Coisotropic submanifolds and their characteristic foliations have been studied extensively in symplectic topology.  The various rigidity properties that they exhibit have been of particular interest. For example, in \cite{Ginzburg1} Ginzburg initiated a program for studying rigidity of coisotropic intersections. In this paper, we prove that coisotropic submanifolds, along with their characteristic foliations, are $C^0$--rigid in the spirit of the Gromov--Eliashberg Theorem.

This celebrated theorem states that a diffeomorphism which is a $C^0$--limit of symplectomorphisms is symplectic.  Motivated by this, symplectic homeomorphisms are defined as $C^0$--limits of symplectomorphisms (see Definition \ref{def:sympeo} and Remark \ref{remark:radical-definition}). Area preserving homeomorphisms, and their products, are examples of symplectic homeomorphisms.  Here is our main result.  

\begin{theo}\label{theo:smooth-C0-coiso-is-coiso}
Let $C$ be a smooth  coisotropic submanifold of a symplectic manifold $(M,\omega)$. Let $U$ be an open subset of $M$ and $\theta\co U \rightarrow V$ be a symplectic homeomorphism. If $\theta (C\cap U)$ is smooth, then it is coisotropic.  Furthermore, $\theta$ maps the characteristic foliation of $C\cap U$ to that of $\theta (C\cap U)$.
\end{theo}
An important feature of the above theorem is its locality: $C$ is not assumed to be necessarily closed and $\theta$ is not necessarily globally defined.  Here is an immediate, but surprising, consequence of Theorem \ref{theo:smooth-C0-coiso-is-coiso}.
\begin{corol}
If the image of a coisotropic submanifold via a symplectic homeomorphism is smooth, then so is the image of its characteristic foliation.
\end{corol}

Theorem \ref{theo:smooth-C0-coiso-is-coiso} uncovers a link between two previous rigidity results and demonstrates that they are in fact extreme cases of a single rigidity phenomenon.

One extreme case, where $C$ is a hypersurface, was established by Opshtein \cite{opshtein}.  Clearly, in this case, the interesting part is the assertion on rigidity of characteristics, as the first assertion is trivially true.
 
Lagrangians constitute the other extreme case. When $C$ is Lagrangian, its characteristic foliation consists of one leaf, $C$ itself. In this case the theorem reads: \textit{If $\theta$ is a symplectic homeomorphism and $\theta(C)$ is smooth, then $\theta(C)$ is Lagrangian.} 
In \cite{LS94}, Laudenbach--Sikorav proved a similar result: \textit{Let $L$ be a closed manifold and $\iota_k$ denote a sequence of Lagrangian embeddings $L \rightarrow (M,\omega)$ which $C^0$--converges to an embedding $\iota$. If $\iota(L)$ is smooth, then (under some technical assumptions) $\iota(L)$ is Lagrangian.} 
On one hand, their result only requires convergence of embeddings while Theorem \ref{theo:smooth-C0-coiso-is-coiso} requires convergence of symplectomorphisms. On the other hand, Theorem \ref{theo:smooth-C0-coiso-is-coiso} is local: It does not require the Lagrangian nor the symplectic manifold to be closed. 

 The above discussion raises the following question.
\begin{question}
 What can one say about $C^0$--limits of coisotropic embeddings and their characteristic foliations?
\end{question}

We would like to point out that Theorem \ref{theo:smooth-C0-coiso-is-coiso} is a coisotropic generalization of the Gromov--Eliashberg Theorem. Indeed, it implies that if the graph of a symplectic homeomorphism is smooth, then it is Lagrangian. \\

As we shall see, the proof of Theorem \ref{theo:smooth-C0-coiso-is-coiso} relies on dynamical properties of coisotropic submanifolds. In particular, we use $C^0$--Hamiltonian dynamics as defined by M\"uller and Oh \cite{muller-oh}. To the best of our knowledge, this is one of the first extrinsic applications of this recent, yet promising, theory.

Following \cite{muller-oh}, we call a path of homeomorphisms $\phi^t$  a hameotopy if there exists a sequence of smooth Hamiltonian functions $H_k$ such that the isotopies $\phi^t_{H_k}$ $C^0$--converge to $\phi$ and the Hamiltonians $H_k$ $C^0$--converge to a continuous function $H$ (see Definition \ref{def:hameo}). Then, $H$ is said to generate the hameotopy $\phi^t$, and to emphasize this we write $\phi^t_H$; the set of such generators will be denoted $C^0_\Ham$.  A foundational result of $C^0$--Hamiltonian dynamics is the uniqueness of generators Theorem (see \cite{viterbo2,buhovsky-seyfaddini}) which states that the trivial hameotopy, $\phi^t = \mathrm{Id},$ can only be generated by those functions in $C^0_\Ham$ which solely depend on time (see also Corollary \ref{cor:uniqueness-thm} below).



Let $H\in C^\infty( [0,1] \times M)$. Recall the following two dynamical properties of a coisotropic submanifold $C$: Assume that $C$ is closed as a subset, 

\noindent \textbf{1.} $H|_C$ is a function of time if and only if $\phi_H$ (preserves $C$ and) flows along the characteristic foliation of $C$. By flowing along characteristics we mean that for any point $p\in C$ and any time $t\geq 0$, $\phi_H^t(p)\in\m F(p)$, where $\m F(p)$ stands for the characteristic leaf through $p$.  

\noindent \textbf{2.}  For each $p \in C, \; H|_{\m F (p)}$ is a function of time if and only if the flow $\phi_H$ preserves $C$.  

We will show that the above two properties hold for continuous Hamiltonians.   The $C^0$--analog of the first property, stated below, plays an important role in the proof of Theorem \ref{theo:smooth-C0-coiso-is-coiso}.

\begin{theo}\label{theo:coiso-unique}
  Denote by $C$ a connected coisotropic submanifold of a symplectic manifold $(M, \omega)$ which is closed as a subset\footnote{  It is our convention that submanifolds have no boundary.  Note that a submanifold is closed as a subset if and only if it is properly embedded. } of $M$. Let $H\in C^0_\Ham$ with induced hameotopy $\phi_H$. The restriction of $H$ to $C$ is a function of time if and only if $\phi_H$ preserves $C$ and flows along the leaves of its characteristic foliation.
\end{theo}
This result answers a question raised by Buhovsky and Opshtein who asked if the above holds in the particular case where $C$ is a smooth hypersurface.  It also drastically generalizes the aforementioned uniqueness of generators Theorem.  Indeed, if $C$ is taken to be $M$, then the characteristic foliation consists of the points of $M$ and the theorem follows immediately:
\begin{corol}\label{cor:uniqueness-thm}
  $H \in C^0_\Ham$ is a function of time if and only if $\phi^t_H = \mathrm{Id}$.
\end{corol}

After the first draft of this article was written, we were asked by Opshtein if the second of the aforementioned properties holds for $C^0$ Hamiltonians.  Our next result provides an affirmative answer to  Opshtein's question.
\begin{theo}\label{theo:emman_question}
  Denote by $C$ a connected coisotropic submanifold of a symplectic manifold $(M, \omega)$  which is closed as a subset of $M$. Let $H\in C^0_\Ham$ with induced hameotopy $\phi_H$.  The restriction of $H$ to each leaf of the characteristic foliation of $C$  is a function of time if and only if the flow $\phi_H$ preserves $C$.
\end{theo}

When $C$ is a Lagrangian, Theorems \ref{theo:coiso-unique} and \ref{theo:emman_question} coincide and both state that: \textit{The restriction of $H$ to $L$ is a function of time if and only if $\phi_H^t(L)=L$ for all $t$.} In an interesting manifestation of Weinstein's creed, ``Everything is a Lagrangian submanifold!'', the general case of Theorems \ref{theo:smooth-C0-coiso-is-coiso}, \ref{theo:coiso-unique} and \ref{theo:emman_question} will be essentially deduced from the a priori particular case of Lagrangians.  \\

The results of this paper establish $C^0$--rigidity of coisotropic submanifolds together with their characteristic foliations. It would be interesting to see if isotropic or symplectic submanifolds exhibit similar rigidity properties: If a smooth submanifold is the image of an isotropic (respectively symplectic) submanifold under a symplectic homeomorphism, is it isotropic (respectively symplectic)? Note that if in these questions one considers, instead of symplectic homeomorphisms, $C^0$--limits of isotropic (respectively symplectic) embeddings then Gromov's results on the $h$-principle provide negative answers in general. In short, isotropic and symplectic embeddings are not $C^0$--rigid. (See \cite[Section 3.4.2]{gromov}, or \cite[Theorems 12.1.1 and 12.4.1]{eliashberg-mishachev}.) 

\subsection*{Defining $C^0$--coisotropic submanifolds}
As we will see in Section \ref{sec:C0_coisotropics}, an interesting feature of Theorem \ref{theo:smooth-C0-coiso-is-coiso} is that it allows us to define $C^0$--coisotropic submanifolds along with their $C^0$--characteristic foliations.  Roughly speaking, a $C^0$--coisotropic will be defined to be a $C^0$--submanifold of a symplectic manifold which is locally symplectic homeomorphic to a smooth coisotropic.  The well-definedness of this notion is a consequence of Theorem \ref{theo:smooth-C0-coiso-is-coiso}.  Furthermore, from the same theorem we conclude that a $C^0$--coisotropic submanifold admits a unique $C^0$--foliation which will be referred to as its $C^0$--characteristic foliation.   

As a consequence, we obtain a definition for $C^0$--Lagrangian submanifolds as $C^0$--coisotropic submanifolds of dimension $n$.   Graphs of symplectic homeomorphisms and graphs of $C^0$ 1--forms, closed in the sense of distributions, constitute examples of $C^0$--Lagrangians.  For further details, we refer the interested reader to Section \ref{sec:C0_coisotropics}.

\subsection*{Main tools: Lagrangian spectral invariants}

In order to prove the main results, we use the theory of Lagrangian spectral invariants
. One consequence of this theory is the existence of the spectral distance $\gamma$ on the space of Lagrangians Hamiltonian isotopic to the 0--section in cotangent bundles introduced by Viterbo in \cite{viterbo1}.\footnote{One of the main features of $\gamma$ is that it is bounded from above by Hofer's distance on Lagrangians. In particular, Lemmas \ref{lemm:our-inequality} and \ref{lemm:LiRi-inequality} also hold with $\gamma$ replaced by Hofer's distance.}

More precisely, we establish inequalities comparing $\gamma$ to a capacity recently defined by Lisi--Rieser \cite{LR}. This capacity, which we denote by $\chzrel$, is a relative (to a fixed Lagrangian) version of the Hofer--Zehnder capacity.  We will now define $\chzrel$.
Fix a Lagrangian $L$.  Recall that a Hamiltonian chord of a Hamiltonian $H$, of length $T$, is a path $ x\co [0,T] \rightarrow M$ such that $x(0)$, $x(T) \in L$ and for all $t\in [0,T]$, $\dot{x}(t)=X_H^t(x(t))$. A Hamiltonian is said to be $L$--slow if all of its Hamiltonian chords of length at most 1 are constant. We denote by $\m H(U)$ the set of admissible Hamiltonians, that is, smooth time-independent functions with compact support included in $U$, which are non-negative and reach their maximum at a point of $L$. For an open set $U$ which intersects $L$, the relative capacity of $U$ with respect to $L$ is defined as
$$\chzrel(U;L)=\sup\{\max f\,|\,f\in \m H(U) \text{ $L$--slow}\}\,.$$ 
For instance, if $B$ is the open ball of radius $r$ in $\mathbb{R}^{2n}$ and $\m L_0 = \mathbb{R}^n \times \{0\}$, then $\chzrel(B;\m L_0)=\frac{\pi r^2}{2}$; see \cite{LR}. 

In what follows, we denote by $L_0$ the 0--section of $T^*L$. The first energy-capacity inequality used in this paper is the following:
\begin{lemma}\label{lemm:our-inequality}
Let $L$ be a smooth closed manifold and $U_-$ and $U_+$ be open subsets of $T^*L$, so that $U_\pm\cap L_0\neq\emptyset$. If a compactly supported Hamiltonian $H$ satisfies $H|_{U_\pm}=\pm C_\pm$ with $C_\pm\in \bb R$ so that $C_\pm > \chzrel(U_\pm;L_0)$,  then $\gamma(\phi_H^{1}(L_0),L_0) \geq \min \{\chzrel(U_-;L_0),\chzrel(U_+;L_0)\}$.
\end{lemma}
This is the Lagrangian analog of the energy-capacity inequality proven for the Hamiltonian spectral distance in \cite[Corollary 12]{HLS12}. Then, as in \cite{HLS12}, we will derive a similar inequality for Hamiltonians (not necessarily constant but) with controlled oscillations on $U_\pm$, see Corollary \ref{corol:nrg-capcity-small-osc}.

Lemma \ref{lemm:our-inequality} can also be established on compact manifolds for weakly exact Lagrangians via Leclercq \cite{Leclercq08} and for monotone Lagrangians via Leclercq--Zapolsky \cite{LZ}.\\

The second energy-capacity inequality is due to Lisi--Rieser \cite{LR}. This is a relative version of the standard energy-capacity inequality, see for example Viterbo \cite{viterbo1}. 

\begin{lemma}\label{lemm:LiRi-inequality}
Let $L$ be a smooth closed manifold. Suppose that $U$ is an open subset of $T^*L$, with $L_0\cap U\neq \emptyset$. Assume that $L'$ is a Lagrangian Hamiltonian isotopic to $L_0$ such that $L'\cap U=\emptyset$. Then $\gamma(L',L_0) \geq \chzrel(U;L_0)$.
\end{lemma}

A special case of this specific inequality appears in Barraud--Cornea \cite{barraud-cornea07} and Charette \cite{charette12}. A similar inequality is worked out in Borman--McLean \cite{BMcL}.\\

Finally, we will need an inequality which provides an upper bound for the spectral distance. Let $g$ denote a Riemannian metric on a closed manifold $L$ and denote by $T^*_rL = \{ (q,p) \in T^*L \,|\, \Vert p \Vert_g \leq r \}$ the cotangent ball bundle of radius $r$.   Suppose that $\phi^t_H(L_0) \subset T^*_rL$ for all $t\in [0,1]$.  Viterbo has conjectured \cite{viterbo3} that there exists a constant $C>0$, depending on $g$, such that $\gamma(\phi^1_H(L_0), L_0) \leq C r$.  This conjecture has many important ramifications; see \cite{MVZ12,viterbo3}. Lemma \ref{lemm:haus_cont_gamma} below is a special case of Viterbo's conjecture; a more precise version of the lemma appears in \cite[Theorem 9.7]{Oh3}.
\begin{lemma}\label{lemm:haus_cont_gamma}
Let $L$ be a smooth closed manifold, $\m V$ a proper open subset of $L$, and $V = \pi^{-1}(\m V) \subset T^*L$, where $\pi \co T^*L \rightarrow L$ is the standard projection. There exists $C>0$, depending on the set $\m V$, such that: For all $r>0$, if $H$ is a smooth, compactly supported Hamiltonian on $T^*L$ such that  $H|_V = 0$, and $\phi^t_H(L_0) \subset T^*_rL$ for all $t\in [0,1]$ then
$\gamma(\phi^1_H(L_0), L_0) \leq C r.$
\end{lemma}

\subsection*{Organization of the paper}

In Section \ref{sec:preliminaries}, we review the preliminaries on $C^0$--Hamiltonian dynamics and Lagrangian spectral invariants. 
In Section \ref{sec:proof-prop}, we prove energy-capacity inequalities (Lemmas \ref{lemm:our-inequality} and \ref{lemm:LiRi-inequality}) as well as the upper bound on the spectral distance (Lemma \ref{lemm:haus_cont_gamma}).
In Section \ref{sec:proof-main-theo-Lagr}, we use these inequalities in order to prove localized versions of Theorem \ref{theo:coiso-unique} in the special case of Lagrangians. 
In Section \ref{sec:coisotr-subm-foliat}, we prove Theorems \ref{theo:smooth-C0-coiso-is-coiso}, \ref{theo:coiso-unique}, and \ref{theo:emman_question} using the results of Section \ref{sec:proof-main-theo-Lagr}.
In Section \ref{sec:C0_coisotropics}, we define $C^0$--coisotropic submanifolds and their characteristic foliations.  In the same section, we provide examples of such $C^0$--objects.  

In Appendix \ref{sec:smooth-c0-lagrangian}, we provide relatively simple, and hopefully enlightening, proofs of Theorems \ref{theo:smooth-C0-coiso-is-coiso} and \ref{theo:coiso-unique} in the special case of closed Lagrangians in cotangent bundles. We hope that this appendix will give the reader an idea of the proofs of the main results while avoiding the technicalities of Sections \ref{sec:proof-main-theo-Lagr} and \ref{sec:coisotr-subm-foliat}.

\subsection*{Aknowledgements}
We thank Samuel Lisi and Tony Rieser for sharing their work with us before it was completed and for related discussions. The inspiration for this paper came partly from an unpublished work by Lev Buhovsky and Emmanuel Opshtein, whom we also thank. We are especially grateful to Emmanuel Opshtein for generously sharing his ideas and insights with us through many stimulating discussions. We are also grateful to Alan Weinstein for interesting questions and suggestions. Lastly we would like to thank the anonymous referees for pointing out several inaccuracies and for many helpful suggestions which have improved the exposition.

This work is partially supported by the grant ANR-11-JS01-010-01.  The research leading to these results has received funding from the European Research Council under the European Union's Seventh Framework Programme (FP/2007-2013) / ERC Grant Agreement  307062.




%
%

\section{Preliminaries}
\label{sec:preliminaries}

\subsection{Symplectic and Hamiltonian homeomorphisms}
\label{sec:sympl-hamilt-home}

In this section we give precise definitions for symplectic and Hamiltonian homeomorphisms and recall a few basic properties of the theory of  continuous Hamiltonian dynamics developed by Müller and Oh \cite{muller-oh}.

Given two manifolds $M_1$, $M_2$, a compact subset $K\subset M_1$, a Riemannian distance $d$ on $M_2$, and two maps $f,g \co M_1\to M_2$, we denote  
$$d_K(f,g)=\sup_{x\in K}d(f(x),g(x)).$$
We say that a sequence of maps $f_i \co M_1\to M_2$ $C^0$--converges to some map $f \co M_1\to M_2$, if for every compact subset $K\subset M_1$, the sequence $d_K(f_i,f)$ converges to 0. This notion does not depend on the choice of the Riemannian metric.

\begin{definition}
  \label{def:sympeo} Let $(M_1,\omega_1)$ and $(M_2,\omega_2)$ be symplectic manifolds.  
	 A continuous map $\theta \co U\to M_2$, where $U \subset M_1$ is open,  is called symplectic if it is the $C^0$--limit of a sequence of symplectic diffeomorphisms $\theta_i \co U\to \theta_i(U)$.    

Let $U_1\subset M_1$ and $U_2\subset M_2$ be open subsets. If a homeomorphism $\theta \co U_1\to U_2$ and its inverse $\theta^{-1}$ are both symplectic maps, we call $\theta$ a symplectic homeomorphism.
\end{definition}

Clearly, if $\theta$ is a symplectic homeomorphism, so is $\theta^{-1}$. By the Gromov--Eliashberg Theorem a symplectic homeomorphism which is smooth is a symplectic diffeomorphism. 

\begin{remark}\label{remark:radical-definition} More generally, one can define a symplectic homeomorphism to be a homeomorphism which is locally a $C^0$--limit of symplectic diffeomorphisms. For simplicity, we do not use this more general definition, however it is evident from the proof of Theorem \ref{theo:smooth-C0-coiso-is-coiso} that it does hold for such homeomorphisms as well.
\end{remark}

We now turn to the definition of Hamiltonian homeomorphisms (called hameomorphisms) introduced by Müller and Oh \cite{muller-oh}.

\begin{definition}
  \label{def:hameo}Let $(M,\omega)$ be a symplectic manifold and $I\subset\R$ an interval.
An isotopy $(\phi^t)_{t\in I}$ is called a hameotopy if there exist a compact subset $K\subset M$ and a sequence of smooth Hamiltonians $H_i$ supported in $K$ such that:
\begin{enumerate}
 \item The sequence of flows $\phi_{H_i}^t$ $C^0$--converges to $\phi^t$ uniformly in $t$ on every compact subset of $I$,
 \item the sequence $H_i(t,\cdot)$ $C^0$--converges to a continuous function $H(t,\cdot)$ uniformly in $t$ on every compact subset of $I$.
\end{enumerate}
We say that $H$ generates $\phi^t$, denote $\phi^t=\phi_H^t$, and call $H$ a continuous Hamiltonian. We denote by $C^0_{\Ham}(M,\omega)$ (or just $C^0_{\Ham}$) the set of all continuous functions $H\co [0,1]\times M\to\R$ which generate hameotopies parametrized by $[0,1]$.
A homeomorphism is called a hameomorphism if it is the time--$1$ map of some hameotopy parametrized by $[0,1]$.
\end{definition}

A continuous function $H \in C^0_{\Ham}$ generates a unique hameotopy \cite{muller-oh}. Conversely, Viterbo \cite{viterbo2} and Buhovsky--Seyfaddini \cite{buhovsky-seyfaddini} proved that a hameo\-to\-py has a unique (up to addition of a function of time) continuous generator. 

One can easily check that generators of hameotopies satisfy the same composition formulas as their smooth counterparts. Namely, if $\phi_H^t$ is a hameotopy, then $(\phi_H^t)^{-1}$ is a hameotopy generated by $-H(t,\phi^t_H(x))$; given another hameotopy $\phi_K^t$, the isotopy $\phi_H^t \phi_K^t$ is also a hameotopy, generated by $H(t,x)+K(t,(\phi^t_H)^{-1}(x))$.

Moreover, we will repeatedly use the following simple fact: If $H\in C^0_{\Ham}(V)$ for some open set $V$ in a symplectic manifold $(M,\omega)$ and if $\theta \co U\to V$ is a symplectic homeomorphism, then $H\circ\theta$ belongs to $C^0_{\Ham}(U)$ and generates the hameotopy $\theta^{-1}\phi_H^t\theta$. This, in particular, holds for smooth $H \co [0,1] \times M\to \R$ supported in $V$.

\subsection{Lagrangian spectral invariants}
\label{sec:lagr-spectr-invar}

In \cite{viterbo1}, Viterbo defined Lagrangian spectral invariants on $\bb R^{2n}$ and cotangent bundles via generating functions. Then Oh \cite{Oh97} defined similar invariants via Lagrangian Floer homology in cotangent bundles which have been proven to coincide with Viterbo's invariants by Milinkovi\'c \cite{milinkovic00}. They have been adapted to the compact case by Leclercq \cite{Leclercq08} for weakly exact Lagrangians and Leclercq--Zapolsky \cite{LZ} for monotone Lagrangians. However, for the type of problems which we consider here ($C^0$--convergence of Lagrangians), we can restrict ourselves to Weinstein neighborhoods and thus work only in cotangent bundles. We briefly outline below the construction of these invariants via Lagrangian Floer homology and collect their main properties in this situation. We refer to Monzner--Vichery--Zapolsky \cite{MVZ12} which gives a very nice exposition of the theory. 

Let $L$ be a smooth compact manifold, $L_0$ denote the 0--section in $T^*L$, and $\lambda$ the Liouville 1--form. To a compactly supported smooth time-dependent Hamiltonian $H\in C^\infty_c([0,1]\times T^*L)$ is associated the action functional
\begin{align*}
  \m A_H \co \Omega(T^*L) \rightarrow \bb R \;, \quad \gamma \mapsto \int_0^1 H_t(\gamma(t))\, dt - \int \gamma^*\lambda
\end{align*}
where $\Omega(T^*L)=\{ \gamma\co [0,1] \rightarrow T^*L \,|\, \gamma(0) \in L_0, \; \gamma(1) \in L_0 \}$. The critical points of $\m A_H$ are the chords of the Hamiltonian vector field $X_H$ which start and end on $L_0$. The spectrum of $\m A_H$, denoted by $\Spec(\m A_H)$, consists of the critical values of $\m A_H$. It is a nowhere dense subset of $\bb R$ which only depends on the time--1 map $\phi_H^1$, hence we sometimes denote it by $\Spec(\phi_H^1)$. 

Following Floer's construction, for a generic choice of Hamiltonian function, $\mathrm{crit}(\m A_H)$ is finite and one can form a chain complex $(CF_*(H),\partial_{H,J})$ whose generators are the critical chords and whose differential counts the elements of the 0--dimensional component of moduli spaces of Floer trajectories (i.e pseudo-holomorphic curves perturbed by $H$) which run between the critical chords (with boundary conditions on $L_0$). The differential relies on the additional data of a generic enough almost complex structure, $J$.

This complex is filtered by the values of the action, that is, for $a\in\bb R$ a regular value of $\m A_H$, one can consider only chords of action less than $a$. Such chords generate a subcomplex of the total complex $CF^a_*(H)$ (because the action decreases along Floer trajectories). We denote by $i^a_*$ the inclusion $CF^a_*(H) \rightarrow CF_*(H)$. By considering homotopies between pairs $(H,J)$ and $(H',J')$, one can canonically identify the homology induced by the respective Floer complexes $H_*(CF(H),\partial_{H,J})$ and  $H_*(CF(H'),\partial_{H',J'})$ and by considering $C^2$--small enough Hamiltonian functions, one can see that the resulting object $HF_*(L_0)$ is canonically isomorphic to the singular homology of $L$.

Thus, one can consider spectral invariants associated to any non-zero homology class $\alpha$ of $L$, defined as the smallest action level which detects $\alpha$:
\begin{align*}
  \ell(\alpha;H) = \inf \{ a\in \bb R \,|\, \alpha \in \mathrm{im}(H_*(i^a)) \}
\end{align*}
In what follows we will only be interested in the spectral invariants  associated to the class of a point and the fundamental class which will be respectively denoted by $\ell_-(H)=\ell([\mathrm{pt}];H)$ and $\ell_+(H)=\ell([L];H)$.

These invariants were proven to be continuous with respect to the $C^0$--norm on Hamiltonian functions so that they are defined for any (not necessarily generic) Hamiltonian. Moreover, they only depend on the time--1 map $\phi^1_H$ induced by the flow of $H$; hence they are well-defined on $\Ham^c(T^*L,d\lambda).$\\

Their main properties are collected in the following theorem, which corresponds to \cite[Theorem 2.20]{MVZ12}, except for \textit{(\ref{si:for-slow-ham})} which we prove below. Note that, except for \textit{(\ref{si:cst-on-Lagr})} and \textit{(\ref{si:for-slow-ham})}, these properties already appear in Viterbo \cite{viterbo1}.

\begin{theo} \label{theo:spectral-properties}
 Let $L$ be a smooth closed connected manifold. Let $L_0$ denote the 0--section of $T^*L$. There exist two maps $\ell_\pm \co \Ham^c(T^*L,d\lambda)\rightarrow  \bb R$ with the following properties:
  \begin{enumerate}
  \item For any $\phi\in\Ham^c(T^*L,d\lambda)$, $\ell_\pm(\phi)$ lie in $\mathrm{Spec}(\phi)$.\label{si:crit-val}
  \item $\ell_- \leq \ell_+$.\label{si:ineq-ellmoins-ellplus}
  \item For any two Hamiltonian functions $H$ and $K$, \\
$\int_0^1 \min(H_t-K_t)\, dt \leq  \ell_-(\phi^1_H) - \ell_-(\phi^1_K) \leq \int_0^1 \max(H_t-K_t)\, dt$, and $\int_0^1 \min(H_t-K_t)\, dt \leq  \ell_+(\phi^1_H) - \ell_+(\phi^1_K) \leq \int_0^1 \max(H_t-K_t)\, dt$.\label{si:cont-for-Hofer}
  \item For any $\phi$ and $\phi'\in \Ham^c(T^*L,d\lambda)$, $\ell_+(\phi\phi') \leq \ell_+(\phi) + \ell_+(\phi')$.\label{si:triang-ineq}
  \item For any $\phi\in \Ham^c(T^*L,d\lambda)$, $\ell_\pm(\phi)=-\ell_\mp(\phi^{-1})$.\label{si:duality}
  \item If $H|_{L_0}\leq c$ (respectively $H|_{L_0}\geq c$ or $H|_{L_0} = c$), then $\ell_\pm(\phi^1_H) \leq c$ (respectively $\ell_\pm(\phi^1_H) \geq c$ or $\ell_\pm(\phi^1_H) = c$).\label{si:cst-on-Lagr}
  \item If $f$ is a $L_0$--slow admissible Hamiltonian, then $\ell_+(\phi^1_f)=\max(f|_{L_0})$ and $\ell_-(\phi_f^1)=0$.\label{si:for-slow-ham}
  \item For any $\phi$ and $\phi'\in \Ham^c(T^*L,d\lambda)$ such that $\phi(L_0)=\phi'(L_0)$, $\ell_+(\phi)-\ell_-(\phi)=\ell_+(\phi')-\ell_-(\phi')$.\label{si:lagr-invar}
  \end{enumerate}
\end{theo}

\begin{proof}[Proof of item \textit{(\ref{si:for-slow-ham})}]
Since $f$ is $L_0$--slow, $\mathrm{Spec}(\phi_f^1)$ consists of critical values of $f$ corresponding to critical points lying in $L_0$. Now for all $s\in [0,1]$, since $f$ is autonomous, $\mathrm{Spec}(\phi_f^s)=\mathrm{Spec}(\phi_{sf}^1)=s\cdot \mathrm{Spec}(\phi_f^1)$. Since in cotangent bundles spectral invariants lie in the spectrum regardless of degeneracy of $f$, by continuity of $\ell_\pm$ there exist $p_\pm\in \mathrm{crit}(f)\cap L_0$ such that $\ell_\pm(\phi_f^s)=s\cdot f(p_\pm)$. Now we claim that for small times $s$, $\ell_+(\phi_{sf}^1)$ (respectively $\ell_-(\phi_{sf}^1)$) is the maximum (respectively the minimum) of $sf$ so that $f(p_+)=\max (f)$ (respectively $f(p_-)=\min (f)$) which concludes the proof.

It remains to prove the claim. 
Since $f$ is autonomous, the path of Lagrangians it generates is a geodesic with respect to Hofer's distance for small times, see Milinkovi\'c \cite[Theorem 8]{milinkovic01}. That is, for $s\in (0,1)$ small enough
\begin{align}\label{eq:milinkovic-geodesic-path}
 \max(sf)-\min(sf) =  \mathrm{length}(\{ L_t \}_{t\in [0,s]}) = d_{\mathrm{Hof}}(L_0,L_{s})
\end{align}
with $L_t = \phi_{f}^t(L_0)$. 
(As $f$ reaches its extrema on $L_0$, the extrema of $sf$ over $T^*L$ and $L_0$ coincide and we remove them from the notation.)
On the other hand, by choosing $s$ small enough we can ensure that 
$L_s=\Gamma_{\!dS}$, the graph of the differential of some smooth function $S \co L \rightarrow \bb R$.
Now in general $d_{\mathrm{Hof}}(L_0,\Gamma_{\!dS})\leq \osc (S)$ and by \cite[Proposition 1.(6)]{milinkovic01}: $\osc(S)=\ell_+(\phi_{\tilde S}^1)-\ell_-(\phi_{\tilde S}^1)$ where $\tilde S$ lifts $S$ to $T^*L$ by pullback via the natural projection and cutoff far from the Lagrangian isotopy. 
Since $\phi_{\tilde{S}}^1(L_0)=\Gamma_{\! dS}$, by Property \textit{(\ref{si:lagr-invar})}, we can replace $\phi_{\tilde S}^1$ by $\phi_{sf}^1$ thus 
\eqref{eq:milinkovic-geodesic-path} leads to the inequality $\max(sf)-\min(sf) \leq \ell_+ (\phi_{sf}^1) -  \ell_- (\phi_{sf}^1)$.
Finally, since $\min(sf) \leq \ell_-(\phi_{sf}^1)\leq \ell_+(\phi_{sf}^1)\leq \max(sf)$, we deduce that $\ell_-(\phi_{sf}^1)=\min(sf)$ and $\ell_+(\phi_{sf}^1)=\max(sf)$. 
\end{proof}

In view of Property \textit{(\ref{si:lagr-invar})}, Viterbo (followed by Oh) derived an invariant $\gamma$ of Lagrangians Hamiltonian isotopic to the 0--section, defined as follows.
\begin{definition} For a Hamiltonian diffeomorphism $\phi\in\Ham^c(T^*L,d\lambda)$ we set $\gamma(\phi)=\ell_+(\phi)-\ell_-(\phi)$ and for a Lagrangian $L$ Hamiltonian isotopic to the 0--section, we set $\gamma(L,L_0)=\gamma(\phi)$ for any $\phi\in\Ham^c(T^*L,d\lambda)$ such that $\phi(L_0)=L$.
\end{definition}
From the properties of spectral invariants, it is immediate that for all $\phi$ and $\psi\in \Ham^c(T^*L,d\lambda)$, $0\leq\gamma(\phi\psi)\leq \gamma(\phi) + \gamma(\psi)$, and that $\gamma(\phi)=\gamma(\phi^{-1})$. 
Moreover, $\gamma(L, L_0)=0$ implies $L=L_0$ as proven in \cite{viterbo1}. 

  One of the main properties of $\gamma$ which will be used in what follows is the fact that
\begin{align}
  \label{eq:gamma-bdd-osc-on-L}
  \gamma(\phi_H^1(L_0),L_0) = \gamma(\phi_H^1) \leq \max_{t\in[0,1]} (\osc (H_t|_{L_0}))
\end{align}
where $\displaystyle \osc (H_t|_{L_0}) =  \max_{L_0} (H_t) - \min_{L_0} (H_t)$.  This inequality can be directly derived from Property \textit{(\ref{si:cst-on-Lagr})} of Theorem \ref{theo:spectral-properties}.  Note that this yields $$\displaystyle \forall t_0 \in \bb R, \;\;\gamma(\phi_H^{t_0}(L_0),L_0) \leq t_0 \cdot \max_{t\in [0, t_0]} (\osc (H_t|_{L_0})).$$ 

%
%

\section{Energy-capacity inequalities}
\label{sec:proof-prop}

\subsection{The energy-capacity inequality for Hamiltonians constant on open sets}
\label{sec:hamitl-const-open}

We prove Lemma \ref{lemm:our-inequality} by mimicking the proof of the corresponding inequality in \cite{HLS12} (here for Lagrangians, in the easier world of aspherical objects). Then we prove a corollary which will be used in the proof of the main result.

\begin{proof}[Proof of Lemma \ref{lemm:our-inequality}]
First, assume that $\ell_-(\phi_H^1) \leq 0$.  Then for any admissible $L_0$--slow function with support in $U_+$, $f\in \m H(U_+)$, we define the $1$--parameter family of Hamiltonians $H_s(t,x)=H(t,x)-sf(x)$ with $s\in [0,1]$. Since $H$ is constant on $U_+$, $H_s$ generates $\phi_{H_s}^1=\phi_H^1 \phi_f^{-s}$. By triangle inequality and duality (i.e Properties \textit{(\ref{si:triang-ineq})} and \textit{(\ref{si:duality})}) of spectral invariants, we get
\begin{align}\label{eq:avant-derniere-etape}
  \ell_+(\phi_f^{s})\leq  \ell_+(\phi_{H_s}^{-1})+ \ell_+(\phi_{H}^1)=\ell_+(\phi_{H}^1)-\ell_-(\phi_{H_s}^1) \;.
\end{align}
Then notice that
\begin{align*}
  \mathrm{Spec}(\m A_{H_s})=\mathrm{Spec}(\m A_{H})\cup \{ C_+ -sf(p) \,|\, p\in\mathrm{crit}(f)\cap U_+ \} \;.
\end{align*}
Since for all $p\in \mathrm{crit}(f)\cap U_+$, $sf(p)\leq \max (f)\leq \chzrel(U_+;L_0)< C_+$, 
 there exists $\varepsilon >0$, such that $\mathrm{Spec}_{\varepsilon}(\m A_{H_s})$, defined as $\mathrm{Spec}(\m A_{H_s})\cap (-\infty,\varepsilon)$, does not depend on $s$ and coincides with $\mathrm{Spec}_{\varepsilon}(\m A_{H})$ which is totally discontinuous. Since the map $s\mapsto \ell_-(\phi_{H_s}^1)$ is continuous and maps $0$ to $\mathrm{Spec}_{\varepsilon}(\m A_{H})$, it has to be constant so that $\ell_-(\phi_{H_1}^1)=\ell_-(\phi_{H}^1)$ and, from \eqref{eq:avant-derniere-etape}, Property \textit{(\ref{si:for-slow-ham})} of spectral invariants  immediately leads to
\begin{align*}
\max(f) =  \ell_+(\phi_f^1)\leq  \ell_+(\phi_{H}^1)-\ell_-(\phi_{H}^1)=\gamma(\phi_H^1(L_0),L_0).
\end{align*}
Since this holds for any $L_0$--slow function in $\m H(U_+)$, we get $\gamma(\phi_H^{1}(L_0),L_0) \geq \chzrel(U_+;L_0)$.

Now, assume that $\ell_-(\phi_H^1)\geq 0$ and consider $\bar{H}(t,x)=-H(t,\phi_H^t(x))$. By assumption, $\phi_H^t$ is the identity on $U_-$ and  $\bar{H}|_{U_-}=-H|_{U_-}=C_-$. Since $\bar{H}$ generates $\phi_{\bar{H}}^1=(\phi_H^1)^{-1}$,
\begin{align*}
  \ell_-(\phi_{\bar{H}}^1) =  -\ell_+(\phi_H^1)  \leq -\ell_-(\phi_H^1) \leq 0.
\end{align*}
Then the first case gives that $\gamma(\phi_{\bar{H}}^{1}) \geq \chzrel(U_-;L_0)$. Since $\gamma(\phi_{\bar{H}}^{1})=\gamma(\phi_H^1)=\gamma(\phi_H^{1}(L_0),L_0)$, we get that when $\ell_-(\phi_H^1)\geq 0$, $\gamma(\phi_{H}^{1}) \geq \chzrel(U_-;L_0)$.

So $\gamma(\phi_{H}^{1}) \geq \min\{\chzrel(U_-;L_0),\chzrel(U_+;L_0)\}$ regardless the sign of $\ell_-(\phi_H^1)$ which concludes the proof.
\end{proof}

From this (and as in \cite{HLS12}) we infer the same result but for Hamiltonian functions which are allowed to have controlled oscillations on $U_\pm$. 
\begin{corol}\label{corol:nrg-capcity-small-osc} 
Let $L$ be a smooth closed manifold and $U_\pm$ be open subsets of $T^*L$ such that $U_\pm\cap L_0\neq \emptyset$. Let $H$ be a Hamiltonian so that for all $t \in [0,1]$
  \begin{enumerate}
    \item $H_t|_{U_+} > \chzrel(U_+;L_0)$, and $H_t|_{U_-} < -\chzrel(U_-;L_0)$, 
    \item $\osc (H_t|_{U_\pm}) < \varepsilon$, for some $\varepsilon >0$.
 \end{enumerate}
 Then, $\gamma(\phi_H^1(L_0),L_0) \geq \min\{\chzrel(U_-;L_0),\chzrel(U_+;L_0)\} - 2 \varepsilon$.
\end{corol}

\begin{proof}
  Fix $\eta >0$. We choose disjoint open subsets $V_\pm$ such that $U_\pm\Subset V_\pm$ (with $\Subset$ denoting compact containment) and\footnote{In the next few lines we loosely denote $\max_{t\in [0,1]} \osc(f_t|_U)$ by $\osc_U(f)$ for readability.} $\osc_{V_\pm}(H) < \osc_{U_\pm}(H) + \eta$. We also choose cut-off functions $\rho_\pm$ with support in $V_\pm$, such that $0\leq \rho_\pm\leq 1$ and $\rho_\pm|_{U_\pm}=1$. Then we define 
\begin{align*}
  h = H - \rho_+(H-C_+) - \rho_-(H+C_-) 
\end{align*}
with $C_+=\inf(H|_{[0,1]\times U_+})$ and $C_-=-\sup(H|_{[0,1]\times U_-})$. By triangle inequality, we get $\gamma(\phi_H^1) \geq   \gamma(\phi_h^1) - \gamma((\phi_H^1)^{-1}\phi_h^1)$ and we now bound the right-hand side terms.

First, notice that $h$ satisfies the requirements of Lemma \ref{lemm:our-inequality}: $h|_{U_\pm} = \pm C_\pm$ with by assumption $C_+  > \chzrel(U_+;L_0), \; \mbox{and} \; C_- > \chzrel(U_-;L_0)$. Thus we immediately get that $\gamma(\phi_h^1)\geq \min\{\chzrel(U_-;L_0),\chzrel(U_+;L_0)\}$. Now by Property \textit{(\ref{si:cont-for-Hofer})} of spectral invariants, 
\begin{align*}
  \gamma((\phi_H^1)^{-1}\phi_h^1) \leq \osc_{T^*L}\big((H-h)\circ (\phi_H)^{-1} \big) \leq \osc_{V_-\cup V_+} (H-h)
\end{align*}
so that 
\begin{align*}
\gamma((\phi_H^1)^{-1}\phi_h^1) 
&\leq  \osc_{V_+} \big(\rho_+(H-C_+)\big) +  \osc_{V_-} \big(\rho_-(H+C_-)\big)   \\
&\leq  \osc_{V_+} (H) +  \osc_{V_-} (H)  \leq  \osc_{U_+} (H) + \osc_{U_-} (H) + 2\eta \\
&\leq 2 \varepsilon + 2\eta .
\end{align*}
Finally, $\gamma(\phi_H^1(L_0),L_0) \geq \min\{\chzrel(U_-;L_0),\chzrel(U_+;L_0)\} - 2 \varepsilon - 2 \eta$ for any $\eta >0$.
\end{proof}

\subsection{The energy-capacity inequality for Lagrangians displaced from an open set}
\label{sec:lagr-disppl-from}

We give a proof of Lemma \ref{lemm:LiRi-inequality} from Lisi--Rieser \cite{LR}, for the reader's convenience. The method of proof is now classical and goes back to Viterbo \cite{viterbo1} (see also Usher's proof of the analogous result \cite{usher} for compact manifolds, itself heavily influenced by Frauenfelder--Ginzburg--Schlenk \cite{FGS}). Recall that in cotangent bundles the spectral invariants only depend on the endpoints of Hamiltonian isotopies; this drastically simplifies the proof.

\begin{proof}[Proof of Lemma \ref{lemm:LiRi-inequality}]
  Assume that $\phi_H^1(L_0)\cap U =\emptyset$. Choose a $L_0$--slow function $f$ in $\m H(U)$. For $s\in [0,1]$, consider the Hamiltonian diffeomorphism $\phi_{sf}^1\phi_H^1$. It is the end of the isotopy defined as the concatenation
\begin{align*}
  \theta^t = \left\{
    \begin{array}{cl}
      \phi_H^{2t}\,, & {t\in [0,1/2]},\\
      \phi_{sf}^{2t-1}\phi_H^1\,, & {t\in [1/2,1]}.
    \end{array}
\right.
\end{align*}
A Hamiltonian chord of $\theta$ is a path $t\mapsto \gamma(t)$ such that $\gamma(0)\in L_0$, $\gamma(1)\in L_0$ and for all $t$, $\gamma(t)=\theta^t(\gamma(0))$. In particular, for such a chord, $\phi_{sf}^{1}\phi_H^1(\gamma(0)) \in L_0$. However, since by assumption $\phi_H^1(L_0) \cap \mathrm{supp}(sf)=\emptyset$, necessarily $\phi_H^1(\gamma(0))$ is not in the support of $f$ and $\gamma(t)=\phi_H^{2t}(\gamma(0))$ for all $t\leq 1/2$ and remains constant $\gamma(t)=\phi_H^1(\gamma(0))$ for $t\geq 1/2$.

This means that for all $s$, the set of Hamiltonian chords remains constant and so does $\mathrm{Spec}(\phi_{sf}^1\phi_H^1)$. Since this set is nowhere dense and $\ell_+$ is continuous (and takes its values in the action spectrum), the function $s\mapsto \ell_+(\phi_{sf}^1\phi_H^1)$ is constant so that $\ell_+(\phi_H^1)=\ell_+(\phi_{f}^1\phi_H^1)$. Thus,
\begin{align*}
  \ell_+(\phi_f^1) \leq \ell_+(\phi_{f}^1\phi_H^1) + \ell_+((\phi_H^1)^{-1}) = \ell_+(\phi_H^1)-\ell_-(\phi_H^1) = \gamma(\phi_H^1)
\end{align*}
by Properties \textit{(\ref{si:triang-ineq})} and \textit{(\ref{si:duality})} of spectral invariants. Since this holds for any $L_0$--slow function $f\in \m H(U)$ and since for such a function $\max(f) = \ell_+(\phi_f^1)$ by Property \textit{(\ref{si:for-slow-ham})} of spectral invariants, the result follows.
\end{proof}

\begin{remark} \label{LiRi_reform}
Lemma \ref{lemm:LiRi-inequality} can be reformulated as follows:

\textit{Let $L, L_0,$ and $U$ be as in Lemma \ref{lemm:LiRi-inequality}.  Suppose that $H$ is a smooth Hamiltonian such that $L_0\cap  \phi^1_H(U) = \emptyset$. Then, $\gamma(\phi^1_H(L_0), L_0) \geq c_\mathrm{LR}(U; L_0)$.}

The equivalence of the above statement to Lemma \ref{lemm:LiRi-inequality} follows easily from the fact that $\gamma(\phi^1_H(L_0), L_0) = \gamma((\phi^1_H)^{-1}(L_0), L_0)$.
\end{remark}

\begin{remark} \label{osc_LiRi}
Let $L, L_0,$ and $U$ be as in Lemma \ref{lemm:LiRi-inequality}.  Suppose that $H$ is a smooth Hamiltonian such that either  $\phi^1_H(L_0) \cap U = \emptyset$ or $L_0\cap  \phi^1_H(U) = \emptyset$.  Then, it follows from Lemma \ref{lemm:LiRi-inequality}, the previous remark, and Inequality \eqref{eq:gamma-bdd-osc-on-L} that
   $\max_{t\in[0,1]} (\osc (H_t|_{L_0})) \geq c_\mathrm{LR}(U;L_0).$

\end{remark}

\subsection{An upper bound for the spectral distance}
\label{sec:haus_cont_gamma}
In this section we prove Lemma \ref{lemm:haus_cont_gamma} which establishes Viterbo's conjecture in a special case.  The fact that Viterbo's conjecture holds under the additional assumptions of Lemma \ref{lemm:haus_cont_gamma} seems to be well known to experts; we provide a proof here for the sake of completeness. 

\begin{proof}[Proof of Lemma \ref{lemm:haus_cont_gamma}] 
Note that since $\phi^t_H(L_0) \subset T^*_rL$ for all $t\in [0,1]$, modifying $H$ outside of $T^*_rL$ leaves $\phi^1_H(L_0)$, and hence $\gamma(\phi^1_H(L_0), L_0)$, unchanged.  Therefore, by cutting $H$ off outside of $T^*_rL$ and replacing $r$ with $2r$ we may assume that $H$ is supported inside $T^*_rL \setminus V$.

Pick $f \co L \rightarrow \mathbb{R}$ to be a Morse function on $L$ whose critical points are all contained in the open set $\m V$.
Because $f$ has no critical points inside $L \setminus \m V$, we may assume, by rescaling, that 
\begin{equation}\label{norm_df}
\Vert df|_{L\setminus \m V} \Vert_g \geq 1.
\end{equation}

Let $\beta \co T^*L \rightarrow \mathbb{R}$ denote a non-negative cutoff function such that $\beta = 1$ on $T^*_RL$, where $R$ is picked so that $R \gg r$. Let $F= \beta \pi^* f \co T^*L \rightarrow \mathbb{R}.$  By picking $R$ to be sufficiently large, we can ensure that, for $t\in [0,2r]$ and $(q,p)$ in a neighborhood of $T^*_rL$, the Hamiltonian flow of $F$ is given by the formula $$\phi^t_F(q,p) = (q, p+tdf(q)).$$
This, combined with (\ref{norm_df}), implies that $\phi^t_F(q,0) \notin T^*_rL$ for $r < t \leq 2r $ and any point $q \in L \setminus \m V$.
Hence, we see that $\phi^t_F(L_0)$ is outside the support of $H$ for $r < t \leq 2r$.  Therefore, 
$\phi^1_H\phi^{2r}_F(L_0) = \phi^{2r}_F(L_0),$ and so 
$$\gamma(\phi^1_H \phi^{2r}_F(L_0), L_0) = \gamma(\phi^{2r}_F(L_0),L_0). $$
Thus we get,
\begin{align*}
 \gamma(\phi_H^1(L_0),L_0)=\gamma(\phi_H^1) &\leq  \gamma(\phi^1_H \phi^{2r}_F) + \gamma((\phi^{2r}_F)^{-1}) =  \gamma(\phi^{2r}_F) + \gamma(\phi^{2r}_F)\\
  &\leq 4r \,\osc(F) = 4r \,\osc(f)
\end{align*}
and the result follows with $C = 4 \, \osc(f).$
\end{proof}

%
%

\section{Localized results for Lagrangians}
\label{sec:proof-main-theo-Lagr}


The main goal of this section is to establish a suitable localized version of Theorem \ref{theo:coiso-unique} for Lagrangians; since we seek localized statements we do not assume that the Lagrangians in question are necessarily closed.   Not surprisingly, in this new setting Theorem  \ref{theo:coiso-unique} does not hold as stated.   The localized results of this section, which have more complicated statements and proofs, are more powerful and they constitute the main technical steps towards proving Theorems \ref{theo:smooth-C0-coiso-is-coiso} and \ref{theo:coiso-unique}.

We prove the analog of the direct implication of Theorem \ref{theo:coiso-unique} in Section \ref{sec:cst-implies-preserv}.  The analog of the converse implication is proven in Section \ref{sec:preserv-implies-cst}.  Since $L$ is a Lagrangian, its characteristic foliation has a single leaf, $L$ itself, and thus in this section we make no mention of characteristic foliations. 

\subsection{$C^0$--Hamiltonians constant on a Lagrangian preserve it}
\label{sec:cst-implies-preserv}
In this subsection, we show that if the restriction of $H \in C^0_\Ham$ to a Lagrangian $L$ is a function of time, then the associated hameotopy $\phi_H$ preserves $L$, locally. More precisely,
\begin{prop}\label{prop:c0lag-constant->preserved}
  Let $L \subset M$ denote a Lagrangian (not necessarily closed) and $H\in C^0_\Ham$ with associated hameotopy $\phi_H$. If $H|_{L}=c(t)$, is a function of time, then for any point $p \in L$ there exists $\varepsilon >0$ such that $\phi_H^t(p)\in L $ for all $t \in [0,\varepsilon]$.
\end{prop}

Our proof of the above proposition will use the following simple lemma on the local structure of Lagrangians.  
\begin{lemma} \label{lemm: local_structure} 
 Let $L \subset M$ denote a Lagrangian.  Around each point $p \in L$ there exists a neighborhood $L_p \subset L$ such that $L_p$ is contained in a closed Lagrangian torus $T \subset M$.
\end{lemma}
\begin{proof}
Let $U$ denote a Darboux ball around $p$, equipped with the standard coordinates $(x_i, y_i),$ such that $L\cap U = \{(x_i, y_i)\,|\, -a < x_i < a, y_i = 0 \}.$  Let $L_p = \{(x_i, y_i)\,|\, -\frac{a}{2} \leq x_i\leq \frac{a}{2}, y_i = 0 \}.$  Then, for all $k$, the projection of $L_p$ onto the $\langle x_k, y_k \rangle$ plane is the segment $[-\frac{a}{2}, \frac{a}{2}] \times \{0\}$ which can be completed to a smooth embedded loop, say $T_k$, in the  $\langle x_k, y_k\rangle$ plane.  Here, we can make sense of the projection onto a coordinate plane by identifying $U$ with a standard ball in $\R^{2n}.$ 
We set $T= T_1 \times \cdots \times T_n$.  This is a Lagrangian torus inside $U$ containing $L_p$.
\end{proof}
\begin{proof}[Proof of Proposition \ref{prop:c0lag-constant->preserved}]
 By replacing $H$ by $H - c(t)$ we can suppose that $H|_{L}= 0$.  Apply Lemma \ref{lemm: local_structure} to obtain $L_p$ and $T$ as described in the lemma and note that by replacing $L$ with $L_p$  we may make the following simplifying assumption: \textit{there exists a Lagrangian torus $T$ in $M$ such that $L \subset T$.}



We will now prove the proposition under this simplifying assumption.  Let $H_i \co [0,1] \times M \to \mathbb{R}$ denote a sequence of smooth Hamiltonians such that $H_i$ converges uniformly to $H$ and $\phi_{H_i}$ converges to $\phi_H$ in $C^0$--topology.  Take $W \Subset V$ to be  open subsets of $M$ such that 
\begin{equation} \label{eq: containments}
p \in W \cap L \quad \text{and}\quad \overline{V} \cap (T \setminus L) = \emptyset.
\end{equation}
 Recall that the symbol $\Subset$ denotes compact containment and $\overline{V}$ denotes the closure of $V$.  The second condition in \eqref{eq: containments} allows us to pick a cutoff function $\beta \co M \to \mathbb{R}$ such that $\beta|_{V} =1$ and $\beta|_{T \setminus L} = 0$.  By shrinking $V$, if needed, we may assume that $\beta$ is supported in a Weinstein neighborhood of $T$.

Let $G_i = \beta H_i $ and  $G = \beta H $.  Observe that $G_i$ converges uniformly to $G$ and $G|_{T} = 0$.  We pick $ \varepsilon > 0 $ such that
\begin{equation} \label{eq: epsilon}
\forall t \in [0, \varepsilon], \;\; \phi^t_H(W) \subset V.
\end{equation}
For $i$ large enough $\phi^t_{H_i}(W) \subset V$ for all $t \in [0, \varepsilon]$.  Since, $G_i|_{V} = H_i|_{V}$  we conclude that, for large $i,$ 
\begin{equation}\label{eq: flows coincide}
\forall (t,x) \in [0, \varepsilon] \times W, \;\; \phi^t_{G_i} =  \phi^t_{H_i}.
\end{equation}

For a contradiction, suppose that  $\phi^{t_0}_H(p)$ is not contained in $L$ for some $ t_0 \in [0, \varepsilon].$  
From (\ref{eq: containments}) and (\ref{eq: epsilon}) we conclude that $\phi^{t_0}_H(p) \notin T$.  Hence, we can find a small ball $B \subset W$ around $p$ which intersects $T$ non-trivially and such that $\overline{\phi^{t_0}_H(B)} \cap T = \emptyset$.  Hence, for $i$ large enough, we have $\phi^{t_0}_{H_i}(B) \cap T= \emptyset.$  From (\ref{eq: flows coincide}) we get that $\phi^{t_0}_{G_i}(B) \cap T = \emptyset.$

We picked $\beta$ such that the Hamiltonians $G_i$ all have support in a Weinstein neighborhood of $T$.  Therefore, we can pass to $T^*T$, apply Lemma \ref{lemm:LiRi-inequality} and its consequences as stated in Remarks \ref{LiRi_reform} and \ref{osc_LiRi} and conclude that $$t_0 \cdot \max_{t \in [0, t_0]}(\osc(G_i(t, \cdot)|_{T_0})) \geq \chzrel(B;T_0),$$ where $T_0$ stands for the 0--section in $T^*T$.  Here, we have used the fact that $\phi^{t_0}_{G_i}$ is the time--1 map of the flow of the Hamiltonian $t_0G_i(t_0t,x)$.   Since $G_i$ converges uniformly to $G$, the same inequality must hold for $G$ but this contradicts the fact that $G|_{T} = 0.$
\end{proof}

\subsection{$C^0$--Hamiltonians preserving a Lagrangian are constant on it}\label{sec:preserv-implies-cst}
In this subsection, we show that if $H\in C^0_\Ham$ generates a hameotopy $\phi_H$ which (locally) preserves a Lagrangian $L$ then, the restriction of $H$ to $L$ is (locally) a function of time. More precisely,

\begin{prop}\label{prop:C0-preserved->constant}
Let $L \subset M$ denote a Lagrangian, $\m U$ an open subset of $L$, and $H\in C^0_\Ham$ with associated hameotopy $\phi_H$. Suppose that $\phi^t_H(\m U) \subset L$ for all $t \in [0,1]$ and let $\m V$ denote the interior of $\cap_{t \in [0,1]} \phi^t_H(\m U)$.  Then, $H(t, \cdot)|_{\m V}$ is a locally constant function for each $t \in [0,1]$. 
\end{prop}
We were recently informed by Y.-G. Oh that it is possible to extract the above proposition from \cite[Theorem 4.9]{Ohlocality}; the techniques of \cite{Ohlocality} are different than ours.

Our proof of the proposition uses the following consequence of Corollary \ref{corol:nrg-capcity-small-osc}.
This result can be viewed as a Lagrangian analog of the uniqueness of generators Theorem \cite[Theorem 2]{HLS12}.  The argument presented here is similar to the proof of the mentioned uniqueness theorem.

\begin{prop}\label{coro:spectral-conv-implies-constant}
  Let $L$ be a smooth closed manifold and $\{H_k\}_k$ a sequence of smooth, uniformly compactly supported  Hamiltonian functions on $T^*L$ (that is, there exists a compact $K$ such that $\cup_k \mathrm{supp}(H_k) \subset K$), so that
  \begin{enumerate}
  \item for all $t\in [0,1]$, $\gamma (\phi^t_{H_k}(L_0),L_0)$ converges to $0$, and
  \item $\{H_k\}_k$ uniformly converges to a continuous function $H$.
  \end{enumerate}
Then, $H$ restricted to $L_0$ is a function of time.
\end{prop}

{The proofs of this section repeatedly use the following simple fact: Let $H$ denote a Hamiltonian, $C^0$ or smooth. The time--$t$ flow of $\tilde{H}(t,x) = aH(t_0 +at,x)$ is given by the expression: $\phi_{\tilde{H}}^t=\phi^{t_0 + at}_H  (\phi_H^{t_0})^{-1}$.}

\begin{proof}[Proof of Proposition \ref{coro:spectral-conv-implies-constant}]
If $H|_{[0,1]\times L_0}$ is not a function of time, there exist $t_0\in [0,1)$ and $x_+$, $x_-\in L_0$ such that $H_{t_0}(x_+)>H_{t_0}(x_-)$. Up to a shift (and cutoff far from $K\cup L_0$), we can assume that $H_{t_0}(x_+) = - H_{t_0}(x_-) = \Delta > 0$.

Now, let $\eps=\Delta/4$ and notice that there exist $\delta\in (0,1]$ and $r>0$ such that $[t_0,t_0+\delta]\subset [0,1]$ and that there exist symplectically embedded balls, centered at $x_\pm$, $B_\pm=\iota_\pm(B_{\bb C^n}(0,r))$, with real part mapped to $L_0$, which are disjoint and such that 
$$\osc_{[t_0,t_0+\delta]\times B_\pm} H<\eps.$$
By shrinking either $\delta$ or $r$, we can assume that the Lisi-Rieser capacity of the balls with respect to $L_0$ satisfy 
$$\chzrel(B_{\pm};L_0)=\delta(\Delta-\eps).$$
Then set $F_k(t,x)=\delta H_k (t_0+\delta t,x)$ and define $F$ accordingly. By construction, the function $F$ satisfies
\[  \forall (t,x)\in [0,1]\times B_+,\ F(t,x) > \chzrel(B_{\pm};L_0),\]
\[  \forall (t,x)\in [0,1]\times B_-,\ F(t,x) <-\chzrel(B_\pm;L_0)\]
and $$ \osc_{[0,1]\times B_\pm}(F)<\delta\eps.$$
Since the Hamiltonians $F_k$ converge uniformly to $F$, they satisfy the same inequalities for $k$ large enough.
Thus, by Corollary \ref{corol:nrg-capcity-small-osc}, $$\gamma(\phi_{F_k}^1(L_0),{L_0}) \geq \chzrel(B_\pm;L_0)-2\delta\eps=\delta(\Delta-3\eps)=\tfrac14\delta\Delta>0.$$ 
Hence,  $\gamma(\phi_{F_k}^1(L_0),{L_0})$ is uniformly bounded away from 0.
 However, $\phi_{F_k}^1=\phi_{H_k}^{t_0+\delta} (\phi_{H_k}^{t_0})^{-1}$ so that, by Properties \textit{(\ref{si:triang-ineq})} and \textit{(\ref{si:duality})} of spectral invariants,
\begin{align*}
  \gamma(\phi_{F_k}^1(L_0),{L_0})=\gamma(\phi_{F_k}^1) &\leq \gamma(\phi_{H_k}^{t_0+\delta})+\gamma((\phi_{H_k}^{t_0})^{-1}) = \gamma(\phi_{H_k}^{t_0+\delta})+\gamma(\phi_{H_k}^{t_0}) \\
&\leq \gamma(\phi_{H_k}^{t_0+\delta}(L_0),{L_0})+\gamma(\phi_{H_k}^{t_0}(L_0),{L_0})
\end{align*}
which goes to 0 when $k$ goes to infinity because of Assumption \textit{(1)} and we get a contradiction.
\end{proof}

\begin{proof}[Proof of Proposition \ref{prop:C0-preserved->constant}]
Assume, for a contradiction, that the conclusion of the proposition fails to hold.  We can therefore find $p \in \m V$ and $t_0 \in [0,1)$ such that $H(t_0, \cdot)$ is not constant on any neighborhood of $p$ in $\m V$. First, note that, up to time reparametrization, we may assume that $t_0 = 0$.  Indeed, replace $H$ with $\tilde{H}(t,x) = aH(t_0 + at,x)$, where $a = 1- t_0,$ and $\m U$ with $\tilde{\m U} = \phi^{t_0}_H(\m U)$.  Then, $\m V$ is contained in the interior of $\cap_{t \in [0,1]} \phi^t_{\tilde{H}}(\tilde{\m U} )$ and  $\tilde{H}(0, \cdot)$ is not constant on any neighborhood of $p$.

Apply Lemma \ref{lemm: local_structure} to obtain $L_p$ and $T$ as described in the lemma.  By shrinking $L_p$, if needed, we may assume that $L_p \Subset \m V.$   Let $U \subset M$ denote a small open set around $p$ which is contained in a Weinstein neighborhood of $T$ and such that $L_p \cap U  = L \cap U = T \cap U \Subset L_p.$  Furthermore, towards the end of this proof we will need to apply Lemma \ref{lemm:haus_cont_gamma}, and so we pick $U$ such that the projection of $U$ to $T$ along the cotangent fibers in the Weinstein neighborhood is a proper subset of $T$ whose complement contains a ball.

Since $L_p \cap U \Subset L_p \Subset \m V$, it follows that there exists a  small $\varepsilon > 0$ such that $\phi^t_H( L_p \cap U), (\phi^t_H)^{-1} (L_p \cap U) \Subset L_p$ for all $t \in [0,\varepsilon].$  Replacing $H$ with $\tilde{H}(t,x) = \eps H(\eps t,x)$,  we may assume that $$\phi^t_H( L_p \cap U),  (\phi^t_H)^{-1} (L_p \cap U) \Subset L_p \text{ for all }t \in [0,1].$$

Next, let $B$ denote an open neighborhood of $p$ which is compactly contained in $U$. Once again, as in the previous paragraph, by a reparametrization in time, where $H$ is replaced with $\tilde{H}(t,x) = \eps H(\eps t,x)$ for a sufficiently small $\eps$,  we may assume that $\phi^t_H(B)$, $(\phi^t_H)^{-1} (B) \Subset U$ for all $t \in [0,1].$

 Pick $q \in B$ such that $H(0, p) \neq H(0,q)$ and take a symplectomorphism $\psi$ supported in $B$ such that $\psi$ preserves $T$ and $\psi(p) = q$. Consider the continuous Hamiltonian $G = (H \circ \psi - H)\circ \phi_H$.  It is supported in $\cup_{t \in [0,1]} (\phi_H^t)^{-1}(B) \Subset U$, and moreover, the flow of $G$ is $(\phi_H^t)^{-1}\psi^{-1}\phi_H^t\psi$.  We will now prove that this flow preserves $T$ globally.  Note that the flow is supported in $U$ and pick $x \in T \cap U \subset L_p$.  Since $\phi^t_H(T\cap U) \subset T$, and $\psi(x) \in T\cap U$, we see that  $\phi^t_H\psi(x) \in T$. First, suppose that $\phi^t_H\psi(x) \notin B$.  Then, $\phi^t_H \psi(x)$ is outside the support of $\psi^{-1}$ and so $(\phi_H^t)^{-1}\psi^{-1}\phi_H^t\psi(x) =  \psi(x)$ which is in $T$.   Next, suppose that $\phi^t_H\psi(x) \in B \cap T$. Then, $\psi^{-1} \phi^t_H\psi(x) \in B\cap T$, and so it suffices to check that $(\phi_H^t)^{-1}(B \cap T) \subset T$: this is because $B\cap T \subset U \cap T \subset U \cap L_p$ and   $(\phi^t_H)^{-1} (L_p \cap U) \Subset L_p$ for all $t \in [0,1]$.  We have proven that the flow of $G$ preserves $T$ globally.

Note that $G|_{T \cap U}$ is not a function of time only: $G(0,p) = H(0,q) - H(0,p) \neq 0$ and $G=0$ near the boundary of $U$.  Hence, we have obtained a $C^0$--Hamiltonian $G$, supported in $U$, whose flow $\phi^t_G$ preserves $T$ globally, but $G(0,\cdot)$ is not constant on $T$.  Because $G \in C^0_{\Ham}$ there exist smooth Hamiltonians $G_i$  such that $\{G_i\}$ converges uniformly to $G$ and $\{\phi_{G_i}\}$ converges to $\phi_G$.  Furthermore, we can ensure that all $G_i$'s are supported in $U$.  This can be achieved by picking a corresponding sequence of smooth Hamiltonians $H_i$ for $H$ and defining $G_i = (H_i \circ \psi - H_i)\circ \phi_{H_i}$.  For large $i$, $G_i$ is supported in $U$. 

Since $U$  is contained in a Weinstein neighborhood of $T$, we can pass to $T^*T$ and work with the Lagrangian spectral invariants of the 0--section $T_0$ associated to the Hamiltonians $G_i$.   Recall that in the second paragraph of the proof we picked the set $U$ so that Lemma \ref{lemm:haus_cont_gamma} could be applied.   For any fixed $ r >0$,  because $\phi^t_G$ preserves $T_0,$  we have $\phi^t_{G_i}(T_0) \subset T^*_rT$ for sufficiently large $i$.   The Hamiltonians $G_i$ are all supported in $U$ and hence using Lemma \ref{lemm:haus_cont_gamma} we conclude that $\gamma(\phi^1_{G_i}(T_0), T_0) \leq C r $, i.e. $\gamma(\phi^1_{G_i}(T_0), T_0) \to 0$.    Of course, by the same reasoning we obtain that $\gamma(\phi^t_{G_i}(T_0), T_0) \to 0$ for all $t \in [0,1]$. Then, Proposition \ref{coro:spectral-conv-implies-constant} implies that $G|_T = c(t)$, which contradicts the fact that $G|_T$ is not a function of time only.
\end{proof}

\section{$C^0$--rigidity of coisotropic submanifolds and their characteristic foliations}
\label{sec:coisotr-subm-foliat}

This section is devoted to the proofs of Theorems \ref{theo:smooth-C0-coiso-is-coiso}, \ref{theo:coiso-unique} and \ref{theo:emman_question}. We begin by proving Theorem \ref{theo:coiso-unique} and then deduce Theorems \ref{theo:smooth-C0-coiso-is-coiso} and  \ref{theo:emman_question} from it.

Before going into the proof, recall (see \cite[Proposition 13.7]{Libermann-Marle87} and \cite{gotay82}) that coisotropic submanifolds admit \emph{coisotropic charts}, that is, for every point $p\in C$, there is a pair  $(\theta,U)$ where $U$ is an open neighborhood of $p$ and $\theta \co U\to V\subset \R^{2n}$ is a symplectic diffeomorphism which maps $p$ to 0 and $C$ to the standard coisotropic linear subspace \begin{align*}
    \m C_0 = \{ (x_1,\ldots, x_n, y_1,\ldots, y_{n}) \,|\, (y_{n-k+1},\ldots, y_{n}) =(0, \ldots,0)  \} .
  \end{align*}
Such a diffeomorphism sends the characteristic foliation of $C$ to that of $\m C_0$, whose leaf through a point 
$q=(a_1,\ldots, a_n, b_1,\ldots, b_{n-k}, 0,\ldots, 0)\in \m C_0$ is the affine subspace 
\begin{align*}
\mathcal{F}_0(q)=\{(a_1,\ldots, a_{n-k},x_{n-k+1},\ldots,x_n, b_1,\ldots, &b_{n-k}, 0,\ldots, 0) \\
&\,|\,(x_{n-k+1},\ldots,x_{n})\in\R^{k}\}  \,.
\end{align*}

The first step of the proof, is establishing the next lemma which is a version of the first implication of Theorem \ref{theo:coiso-unique} that does not require the coisotropic submanifold to be a closed subset but holds only for small times.

\begin{lemma}\label{lemm:foliat-loc-preserved}
  Let $(M,\omega)$ be a symplectic manifold and $C$ a coisotropic submanifold of $M$. Let $H\in C^0_\Ham(M,\omega)$ with induced hameotopy $\phi_H$. Assume that the restriction of $H$ to $C$ only depends on time. Then, for every $p\in C$, there exists $\eps>0$ such that for all $t\in[0,\eps]$, $\phi_H^t(p)$ belongs to $\mathcal{F}(p)$, the characteristic leaf of $C$ through $p$. 
\end{lemma}

Before going into the details of the proof of Lemma \ref{lemm:foliat-loc-preserved}, we make the following observation. The lemma holds for coisotropic submanifolds of arbitrary codimension but its proof will follow from the particular case of Lagrangians. As mentioned in the introduction, this is not surprising in view of Weinstein's creed: ``Everything is a Lagrangian submanifold!'' \cite{weinstein}. 

\begin{proof} Let $p\in C$ and let $(U,\theta)$ be a coisotropic chart as defined above. 
For $i\in\{1,\ldots,n-k\}$ consider the Lagrangian linear subspaces 
$$\Lambda_i=\{(x_1,\ldots,x_n,y_1,\ldots,y_n)\,|\,x_i=0\text{ and }\forall j\neq i, y_j=0\},$$ and their pull backs $L_i=\theta^{-1}(\Lambda_i)$. Clearly, for all $i\in\{1,\ldots, n-k\}, L_i\subset  C\cap U$ and 
$$\m F(p)\cap U=\bigcap_{i=1}^{n-k} L_i.$$
Let $H$ be as in the statement of Lemma \ref{lemm:foliat-loc-preserved}. Then for any $i$, the restriction of $H$ to $L_i$ is a function of time since $L_i$ is included in $C$. Thus by Proposition \ref{prop:c0lag-constant->preserved} there exists $\eps_i>0$ such that for all $t\in[0,\eps_i]$, $\phi_H^t(p)\in L_i$. Taking $\eps=\min\{\eps_1,\ldots,\eps_{n-k}\}$, we get 
$$\forall t\in[0,\eps], \; \phi_H^t(p)\in \bigcap_{i=1}^{n-k} L_i\subset \m F(p).$$ 
\end{proof}

We can now prove Theorem \ref{theo:coiso-unique}.

\begin{proof}[Proof of Theorem \ref{theo:coiso-unique}] Let $H\in C^0_\Ham$ such that $H|_{C}$ is a function of time only and pick $p\in C$. 
For a contradiction, assume that for some $t>0$, $\phi_H^t(p)\notin\m F(p)$ and set $t_0=\inf\{t>0\,|\,\phi_H^t(p)\notin\m F(p)\}$. Note that since $C$ is a closed subset, the point $\phi_H^{t_0}(p)$ belongs to $C$.  Then, consider the Hamiltonian $K_t=-H_{-t+t_0}$, so that $\phi_K^t=\phi_H^{-t+t_0} (\phi_H^{t_0})^{-1}$. Its restriction to $C$ is also a function of time.
 Lemma \ref{lemm:foliat-loc-preserved} applied to $K$ at the point $\phi_H^{t_0}(p)$ implies that for some small $t>0$, $\phi_{H}^{t_0-t}(p)\in \m F(\phi_H^{t_0}(p))$. But by definition of $t_0$, we also have $\phi_{H}^{t_0-t}(p)\in \m F(p)$, hence $\phi_{H}^{t_0}(p)\in \m F(p)$.
Now apply  Lemma \ref{lemm:foliat-loc-preserved} again to $H_{t+t_0}$ at the point  $\phi_{H}^{t_0}(p)$. We get that for some $\eps'>0$  and all $t\in [t_0,t_0+\eps']$, $\phi_H^t(p)\in \m F(p)$ which contradicts the definition of $t_0$. Thus, $\phi_H^t(p)\in\m F(p)$ and the direct implication of Theorem \ref{theo:coiso-unique} follows.

We now prove the converse. Assume that the flow of $H\in C^0_{\Ham}$ preserves each leaf of the characteristic foliation. We are going to show first that the function $H_0$ is locally constant. 

Let $p\in C$ and $\theta \co U\to V$ be a coisotropic chart around $p$, with $\theta(p)=0$. For $\sigma>0$ small enough, the set $\cap_{t\in[0,\sigma]}(\phi_H^t)^{-1}(U)$ contains $p$ in its interior. Denote by $U'$ this interior for some fixed $\sigma$. Similarly, for $s\in(0,\sigma]$ small enough, $\cap_{t\in[0,s]}\phi_H^t(U')$ contains $p$ in its interior. 
Let $U''$ be an open neighborhood of $p$ contained in this interior, and with the property that $\theta(U'')$ is convex. Let $q$ be any other point in $U''$ and $\Lambda$ be a linear Lagrangian subspace included in $\m C_0$, containing $\theta(q)$ and the standard leaf $\m F_0(0)$. The subspace $\Lambda$ can be written as the union of the leaves $\m F_0(x)$ for all $x\in \Lambda$.

Now, consider the Lagrangian $L=\theta^{-1}(\Lambda\cap V)$. Let $\m U=L\cap U'$ and $\m V=L\cap U''$. By construction, $q\in\m V$.
By assumption $\phi_H^t(\m U)\subset L$ for all $t\in[0,s]$. We may apply Proposition \ref{prop:C0-preserved->constant} to $L$ and the continuous Hamiltonian $K_t(x)=sH_{st}(x)$ which generates the hameotopy $\phi_H^{st}$. We get that for any $t\in[0,1]$, $K_t$ is locally constant on $\m V$. Equivalently, for any $t\in[0,s]$, $H_t$ is locally constant on $\m V$.
Now since $\theta(U'')$ is convex and $\Lambda$ is linear, $\theta(U'')\cap\Lambda$ is connected. It follows that $\m V$ is also connected and therefore $H_t(p)=H_t(q)$. To summarize, we proved that for $t$ small enough, $H_t$ is constant on $U''\cap C$. In particular, $H_0$ is locally constant on $C$.

Since $C$ is assumed to be connected, this means that $H_0$ is constant on $C$. The argument we followed for $t=0$ applies for any other initial time. Thus, $H_t$ must be constant on $C$ for any $t$. 
\end{proof}

The proof of Theorem \ref{theo:smooth-C0-coiso-is-coiso} relies on the first implication of Theorem \ref{theo:coiso-unique} and the following characterization of coisotropic submanifolds and their characteristic foliations: 

\emph{A submanifold is coisotropic if and only if the flow of every autonomous Hamiltonian constant on it preserves it. Moreover, the leaf through a point $p$ is locally the union of the orbits of $p$ under the flows of all such Hamiltonians}. 

The next lemma is based on this characterization. 

\begin{lemma}\label{lemm:charact-smooth-coiso} Let $C$ be a submanifold in a symplectic manifold $(M,\omega)$. Assume that every point $p\in C$ admits an open neighborhood $V$ such that any $H\in C_c^\infty(V)$, with $H|_C\equiv 0$, satisfies $\phi_H^t(p)\in C$ for every $t\in[0,+\infty)$. Then $C$ is coisotropic. 

Moreover, for such a neighborhood $V$, there exists a smaller neighborhood $W\Subset V$ such that, the leaf $\m F(p)$ of the characteristic foliation of $C$ passing through $p$ satisfies
$$W\cap \m F(p)=W\cap \{\phi_H^t(p)\,|\,t\in[0,+\infty), H\in C_c^\infty(V), H|_C\equiv 0\}.$$
\end{lemma}

\begin{proof} Let $p\in C$ and let $V$ be an open subset as in the statement of the lemma. Assume that $C$ coincides locally with $f_1^{-1}(0)\cap \ldots\cap f_k^{-1}(0)$ for some smooth functions $f_1,\ldots, f_k$ whose differentials are linearly independent at $p$. By multiplying by an appropriate cutoff function, we can assume that these functions are defined everywhere on $M$, have compact support in $V$, and vanish on $C$.

The Hamiltonian vector fields at $p$ of $f_1,\ldots, f_k$ span $(T_pC)^{\omega}$, and by assumption belong to $T_pC$. Thus $(T_pC)^{\omega}\subset T_pC$ and $C$ is coisotropic.

Now, since the characteristic leaves are preserved by smooth Hamiltonians constant on $C$, we have the inclusion 
 $$\m F(p)\supset \{\phi_H^t(p)\,|\,t\in[0,+\infty), H\in C_c^\infty(V), H|_C\equiv 0\}.$$
Conversely, consider the map $$F \co \R^k\to C,\quad (v_1,\ldots, v_k)\mapsto \phi^1_{\sum_{i=1}^kv_if_i}(p).$$ Since,  $\sum_{i=1}^kv_if_i$ is constant on $C$, its flow preserves the characteristics, hence $F$ takes values in the characteristic leaf $\mathcal F(p)$ through $p$. The partial derivatives of $F$ at $0$ are $\partial_{v_i}F(0)=X_{f_i}(p)$ and in particular they are linearly independent and span $T_p\m F(p)=(T_pC)^\omega$. The inverse function theorem then shows that $F$ is a  diffeomorphism from a neighborhood of 0 to a neighborhood of $p$ in $\mathcal{F}(p)$. This shows $$W\cap \m F(p)\subset W\cap \{\phi_H^t(p)\,|\,t\in[0,+\infty), H\in C_c^\infty(V), H|_C\equiv 0\}$$ for some neighborhood of $p$ in $M$ and finishes the proof of Lemma \ref{lemm:charact-smooth-coiso}.
\end{proof}

We are now ready to prove Theorem \ref{theo:smooth-C0-coiso-is-coiso}.

\begin{proof}[Proof of Theorem \ref{theo:smooth-C0-coiso-is-coiso}] Let $C$ be a smooth coisotropic submanifold, and $\theta \co U\to V$ be a symplectic homeomorphism. Assume $C'=\theta(U\cap C)$ is smooth. Let $p'\in C'$ and $p=\theta^{-1}(p')$. By passing to an appropriate Darboux chart around $p$, we may assume that $U\subset\R^{2n}$, $p=0\in\R^{2n}$, and $C=\m C_0$.
We are going to prove that any function $H\in C_c^\infty(V)$, with $H|_{C'}\equiv 0$, satisfies $\phi_H^t(p')\in C'$ for all $t\in[0,+\infty)$. According to Lemma \ref{lemm:charact-smooth-coiso}, this will imply that $C'$ is coisotropic.

Let $H$ be such a function and consider the function $H\circ\theta$. It is compactly supported in $U$ and can be extended  by 0 outside $U$ to a continuous compactly supported function $K \co \R^{2n}\to \R$. Since $H$ is smooth and $\theta$ is a symplectic homeomorphism, $K\in C^0_{\Ham}(\R^{2n},\omega_0)$. 
Since $K|_{\m C_0}=0$, Theorem \ref{theo:coiso-unique} yields $\phi_K^t(0)\in\m F_0(0)$ for any $t\geq 0$. Since $K$ has support in $U$, we have 
\begin{equation}\label{eq:preserved-leaf}\forall t\geq 0,\ \phi_K^t(0)\in\m F_0(0)\cap U\subset \m C_0\cap U.
\end{equation}
Since $\phi_H^t=\theta \phi_K^t \theta^{-1}$, we deduce $\phi_H^t(p')\in C'$ as desired and hence that $C'$ is coisotropic.  

Denote $\m F'$ the characteristic foliation of $C'$. From (\ref{eq:preserved-leaf}), we deduce that for any $H\in C_c^\infty(V)$, $H|_{C'}\equiv 0$, 
$$\forall t\geq 0,\ \phi_H^t(p')\in \theta(\m F_0(0)\cap U).$$
Now according to Lemma \ref{lemm:charact-smooth-coiso}, there exists a neighborhood $W\subset V$ such that 
$$W\cap \m F'(p')=W\cap \{\phi_H^t(p')\,|\,t\in[0,+\infty), H\in C_c^\infty(V), H|_{C'}\equiv 0\}.$$
Thus,
$$W\cap \m F'(p')\subset W\cap\theta(\m F_0(0)).$$
We get the reverse inclusion by switching the roles of $C$ and $C'$, and we see that $\theta$ sends locally $\m F_0(0)$ onto $\m F'(p')$. 
\end{proof}

Let us now turn to the proof of Theorem \ref{theo:emman_question}. The proof has three main ingredients: the Lagrangian case in Theorem \ref{theo:coiso-unique} (i.e., Propositions \ref{prop:c0lag-constant->preserved} and \ref{prop:C0-preserved->constant}), Theorem \ref{theo:smooth-C0-coiso-is-coiso}, and the fact that the graph of the characteristic foliation, given by $$\Gamma(\m F)=\{(x,x')\in M\times M\,|\,x\in C, x'\in\m F(x)\},$$
is Lagrangian in the product $M\times M$ endowed with the symplectic form $\omega \oplus (-\omega)$, as long as it is a submanifold. 

\begin{proof}[Proof of Theorem \ref{theo:emman_question}] Let $p\in C$ and $(U,\theta)$ be a coisotropic chart of $C$ around $p$ sending $p$ to 0. The symplectic diffeomorphism $\Theta=\theta\times\theta$, defined on $U\times U$ maps $\Gamma(\m F)$ to the graph of the standard characteristic foliation 
  \begin{align*}
\Gamma(\m F_0)=\{(x_1,\ldots,x_n,y_1,\ldots,y_n,x_1',&\ldots,x_n',y_1',\ldots,y_n')\in\R^{2n}\times\R^{2n}\,| \\
\forall i\in\{n-k+1,&\ldots,n\}, y_i=y_i'=0 \text{ and } \\
& \forall j\in\{1,\ldots,n-k\},x_i=x_i', y_i=y_i'\} \,. 
 \end{align*}
Since $\Gamma(\m F_0)$ is a Lagrangian submanifold of $\R^{2n}\times\R^{2n}$, then $\Lambda=\Theta^{-1}(\Gamma(\m F_0))=(U\times U)\cap \Gamma(\m F)$ is a Lagrangian submanifold of $M\times M$. 

Now note that if $H\in C^0_{\Ham}$ then the function $K:[0,1]\times M\times M\to\R$ given by $K_t(x,x')=H_t(x)-H_t(x')$ is a continuous Hamiltonian generating the hameotopy $\phi_H\times\phi_H$ (recall that the symplectic form on $M\times M$ is $\omega\oplus(-\omega)$). 

Assume for a contradiction that $H$ is a function of time on every leaf of the characteristic foliation $\m F$ of $C$ and that for some point $q\in C$ and some time $t<1$, $\phi_H^t(q)\notin C$. Set $t_0=\sup\{t\geq 0\,|\, \phi_H^t(q)\in C\}$. Since $C$ is a closed subset of $M$, $p=\phi_H^{t_0}(q)\in C$ and we can assume that the above construction yielding the construction of $\Lambda$ is performed in the neighborhood of this point. Consider the ''time-reparametrized'' Hamiltonians $\tilde{H}$, $\tilde{K}$ given by $\tilde{H}(t,x)=(1-t_0)H(t_0+(1-t_0)t,x)$ and $\tilde{K}(t,x,x')=\tilde{H}(t,x)-\tilde{H}(t,x')$. 
The fact that $H$ is a function of time on any leaf implies that the restriction of $\tilde{K}$ to $\Gamma(\m F)$ is identically 0. In particular, it vanishes on the Lagrangian $\Lambda$ and according to Proposition \ref{prop:c0lag-constant->preserved} there exists $\eps>0$ such that for all $t\in[0,\eps]$, $\phi_{\tilde{K}}^t(p,p)=(\phi_{\tilde{H}}^t\times\phi_{\tilde{H}}^t)(p,p)\in\Lambda$. This implies that $$\phi_H^{t_0+\eps(1-t_0)}(q)=\phi_{\tilde{H}}^\eps(p)\in C,$$ which contradicts the maximality of $t_0$. 

Conversely, assume that the flow $\phi_H^t$ preserves $C$. By Theorem \ref{theo:smooth-C0-coiso-is-coiso}, $\phi_H^t$ sends leaves to leaves and in particular, it preserves the graph of the foliation. Therefore, for any point $p\in C$, we may apply Proposition \ref{prop:C0-preserved->constant} to the Lagrangian $\Lambda$ and the continuous Hamiltonian $K$. We get that on a neighborhood of $(p,p)$, and for small times $t$, $K_t$ is constant. Since $K_0(p,p)=0$ we get that $K_0$ vanishes in a neighborhood of $(p,p)$. But this implies that $H_0$ is constant on a neighborhood of $p$ in the leaf $\m F(p)$. The argument can be performed for any $p\in C$ and at any initial time $t$ instead of 0. It shows that $H_t$ is locally constant, hence constant, on leaves.
\end{proof}

%
%

\section{Defining $C^0$--Coisotropic submanifolds and their characteristic foliations} \label{sec:C0_coisotropics}

In this section we will use Theorem \ref{theo:smooth-C0-coiso-is-coiso} to define $C^0$--coisotropic submanifolds and their characteristic foliations.  Below, we assume that $\bb R^{2n}$ is equipped with the standard symplectic structure. Recall from the beginning of Section \ref{sec:coisotr-subm-foliat} that every coisotropic submanifold of codimension $k$ is locally symplectomorphic to 
\begin{align*}
    \m C_0 = \{ (x_1,\ldots, x_n, y_1,\ldots, y_{n}) \,|\, (y_{n-k+1},\ldots, y_{n}) =(0, \ldots,0)  \}\subset \R^{2n} \;,
  \end{align*}
and that the leaf of its characteristic foliation, $\m F_0$, passing through $p=(a_1,\ldots, a_n, b_1,\ldots, b_{n-k},0, \ldots,0)$ is given by
\begin{align*}
  \m F_0(p) = \{ (a_1,\ldots, a_{n-k}, x_{n-k+1},\ldots, x_{n}, b_1,\ldots, &b_{n-k},0, \ldots,0)  \,|\, \\
& (x_{n-k+1},\ldots, x_{n}) \in \bb R^{k} \} \;.
\end{align*}

\begin{definition}\label{def: c0-coisotropic}
  A codimension--$k$ $C^0$--submanifold $C$ of a symplectic manifold $(M,\omega)$ is \emph{$C^0$--coisotropic} if around each point $p\in C$ there exists a \emph{$C^0$--coisotropic chart}, that is, a pair $(U, \theta)$ with $U$ an open neighborhood of $p$ and  $\theta \co U \rightarrow V\subset \bb R^{2n}$ a symplectic homeomorphism, such that $\theta (C\cap U)=\m C_0\cap V$.

 A codimension--$n$ \emph{$C^0$--coisotropic} submanifold is called a \emph{$C^0$--Lagrangian.}
\end{definition}

\begin{example}
Graphs of symplectic homeomorphisms are $C^0$--Lagrangians. Graphs of differentials of $C^1$ functions and, more generally, graphs of $C^0$ 1--forms, closed in the sense of distributions, provide a family of non trivial examples; see Proposition \ref{prop:graphs-c0-closed} for a proof.

Conversely, we could ask whether every continuous 1--form whose graph is a $C^0$--Lagrangian is closed in the sense of distributions. An affirmative answer in a particular case appears in Viterbo \cite[Corollary 22]{viterbo2}.
\end{example}

As a consequence of Theorem \ref{theo:smooth-C0-coiso-is-coiso}, $C^0$--coisotropic submanifolds carry  ($C^0$--) characteristic foliations in the following sense.
\begin{prop}\label{coro:c0coiso-have-c0foliation}
Any $C^0$--coisotropic submanifold $C$ admits a unique $C^0$--foliation $\m F$ which is mapped to $\m F_0$ by any $C^0$--coisotropic chart.
\end{prop}
\begin{proof}If such a foliation exists it has to coincide with $\theta^{-1}(\m F_0)$ on the domain of any $C^0$--coisotropic chart $\varphi$. The only thing to check is that for any two $C^0$--coisotropic charts $\theta_1 \co U_1\to V_1$ and $\theta_2 \co U_2\to V_2$, the foliations $\theta_1^{-1}(\m F_0)$ and $\theta_2^{-1}(\m F_0)$ coincide on $U_1\cap U_2$. But this follows immediately from Theorem \ref{theo:smooth-C0-coiso-is-coiso} applied to $C=\m C_0$ and $\theta=\theta_1 \theta_2^{-1} \co \theta_2(U_1\cap U_2)\to \theta_1(U_1\cap U_2)$.\end{proof}

Theorem \ref{theo:smooth-C0-coiso-is-coiso} states that a smooth $C^0$--coisotropic submanifold is coisotropic and its natural $C^0$--foliation coincides with its characteristic foliation.

\begin{example}
  If $C=\theta (C')$, with $C'$ a smooth coisotropic submanifold and $\theta$ a symplectic homeomorphism, then $\m F=\theta (\m F')$ where $\m F'$ is the characteristic foliation of $C'$. 
	
   One may wonder if every topological hypersurface is $C^0$--coisotropic.  It is possible to show, via an application of Proposition \ref{coro:c0coiso-have-c0foliation}, that the boundary of the standard cube in $\R^4$ does not possess a $C^0$--characteristic foliation, and hence, it is not  $C^0$--coisotropic.  
\end{example}


The following proposition tells us that Theorems \ref{theo:coiso-unique} and \ref{theo:emman_question} hold for $C^0$--coisotropic submanifolds.
\begin{prop}\label{prop:C0_coiso-unique}
  Denote by $C$ a connected $C^0$--coisotropic submanifold of a symplectic manifold $(M, \omega)$  which is closed as a subset of $M$. Let $H\in C^0_\Ham$ with induced hameotopy $\phi_H$. 
	\begin{enumerate}
  \item The restriction of $H$ to $C$ is a function of time if and only if $\phi_H$ preserves $C$ and flows along the leaves of its ($C^0$--)characteristic foliation.
	\item The restriction of $H$ to each leaf of the characteristic foliation of $C$  is a function of time if and only if the flow $\phi_H$ preserves $C$. 
	\end{enumerate}
\end{prop}
The above can be proven by adapting the proofs of Theorems \ref{theo:coiso-unique} and \ref{theo:emman_question} to $C^0$--coisotropics.  We will not provide a proof for Proposition \ref{prop:C0_coiso-unique} here, and we only mention that to adapt the proofs one would have to introduce $C^0$--coisotropic charts and use the following simple fact: if $\theta$ is a symplectic homeomorphism and $H\in C^0_{Ham}$ then $H\circ \theta \in C^0_{Ham}$ and $\phi^t_{H\circ \theta} = \theta^{-1} \phi^t_H \theta.$ \\  

Finally, we provide a family of non trivial examples of $C^0$--Lagrangians.

\begin{prop} \label{prop:graphs-c0-closed}Let $\alpha$ be a $C^0$ 1--form on a smooth manifold $N$ which is closed in the sense of distributions. Then, its graph, $\mathrm{graph}(\alpha)\subset T^*N$, is a $C^0$--Lagrangian.
\end{prop}

\begin{proof} Since the statement is local, it is sufficient to prove it when $N$ is an open set in $\R^n$. Then $\alpha$ can be written as $\alpha=\sum_{i=1}^n p_i(x)dx_i$, where $x_1,\ldots,x_n$ are the canonical coordinates in $\R^n$ and $p_1,\ldots,p_n$ continuous functions on $N$. The fact that $\alpha$ is closed is equivalent to the equations 
\begin{equation}\label{eq:closed-in-coordinates}
\forall i,j \in \{1,\ldots, n\}, \; \partial_jp_i=\partial_ip_j,
\end{equation}
where $\partial_ip_j$ is the $i$--th partial derivative of $p_j$ in the sense of distributions.

We use convolution to approximate $\alpha$. To that end, take a compactly supported smooth function $\rho$ such that $\rho\geq 0$, and $\int_N\rho(x)dx=1$ and set $\rho_\eps(x)=\frac1{\eps^n}\rho\!\left(\frac{x}{\eps}\right)$ for every $\eps>0$. For any continuous function $f$ on $N$, the functions 
$$f\ast\rho_\eps(x)=\int_Nf(y)\rho_\eps(x-y)dy$$
are well-defined on any compact subset of $N$ for $\eps$ small enough. Moreover, for any $\eps$, $f\ast\rho_\eps$ is smooth, converges locally uniformly to $f$ as $\eps$ goes to 0, its differential satisfies $d(f\ast\rho_\eps)=(df)\ast\rho_\eps$ and converges in the sense of distributions to $df$.

Let $U\Subset N$ be an open subset of $N$. Then, for $\eps$ small enough, 
$$\alpha_\eps=\sum_{i=1}^np_i\ast \rho_\eps \, dx_i$$
is a well-defined 1--form on $U$. It satisfies Equations (\ref{eq:closed-in-coordinates}) and thus is closed. Moreover, it converges uniformly to $\alpha$ on $U$.

Now let $\phi_\eps$ be the family of symplectic diffeomorphisms of $T^\ast U$ defined by $\phi_\eps(x,p)=(x,p+\alpha_\eps(x))$. They converge uniformly on $U$ to the symplectic homeomorphism $\phi \co T^\ast U\to T^\ast U$, $(x,p)\mapsto(x, p+\alpha(x))$ and $\mathrm{graph}(\alpha)$ restricted to $T^\ast U$ is $\phi(U)$. This shows that $\mathrm{graph}(\alpha)$ is locally the image of a smooth Lagrangian by a symplectic homeomorphism. 
\end{proof}



\appendix
\section{The main results for closed Lagrangian}
\label{sec:smooth-c0-lagrangian}
In this section we provide relatively simple proofs for Theorems \ref{theo:smooth-C0-coiso-is-coiso}, \ref{theo:coiso-unique}, and \ref{theo:emman_question} in an enlightening and important special case. We suppose that $M=T^*L$ equipped with its canonical symplectic structure for some closed smooth manifold $L$. Denote by $\theta$ a symplectic homeomorphism of $T^*L$. And let $L' = \theta(L_0)$, where $L_0$ denotes the 0--section of $T^*L$.  

Below, we will prove Theorems \ref{theo:smooth-C0-coiso-is-coiso}, \ref{theo:coiso-unique}, and \ref{theo:emman_question} in the special case where the coisotropic $C$ is taken to be  the zero section $L_0$.  In this case Theorem \ref{theo:smooth-C0-coiso-is-coiso} states the following:
\begin{theo}\label{theo:special_case_smooth_Lag}
 If $L'$ is smooth, then it is Lagrangian. 
\end{theo}
In the settings considered in this appendix, Theorems \ref{theo:coiso-unique} and \ref{theo:emman_question} coincide and state the following:
\begin{theo}\label{theo:special_case_Lag_uniq}
Let $H\in C^0_{\Ham}$ with induced hameotopy $\phi_H$.  The restriction of $H$ to $L_0$ is a function of time if and only if $\phi_H$ preserves $L_0$.  
\end{theo}

We believe that the above special cases provide the reader with the opportunity to get an idea of the proofs of our main results without having to go through the technical details of Sections \ref{sec:proof-main-theo-Lagr} and \ref{sec:coisotr-subm-foliat}.

We will first show that Theorem \ref{theo:special_case_smooth_Lag} follows from Theorem \ref{theo:special_case_Lag_uniq}.  In order to do so we will need the following dynamical characterizations of isotropic and coisotropic submanifolds, respectively.

\begin{lemma}\label{lem:iso_charact}
Let $I$ denote a (smooth) submanifold of a symplectic manifold $(M, \omega)$.  The following are equivalent: 
\begin{itemize}
\item $I$ is isotropic,
\item For every smooth Hamiltonian $H$, if $\phi_H$ preserves $I$, then $H|_I$ is a function of time only.
\end{itemize}
\end{lemma}

\begin{lemma}\label{lem:coiso_charact}
Let $C$ denote a (smooth) submanifold of a symplectic manifold $(M, \omega)$.  The following are equivalent: 
\begin{itemize}
\item $C$ is coisotropic,
\item  For every smooth Hamiltonian $H$, if $H|_C$ is a function of time only, then $\phi_H$ preserves $C$.
\end{itemize}
\end{lemma}

We leave the proofs of the above lemmas, which follow from symplectic linear algebra, to the reader.  In the proof of Theorem \ref{theo:smooth-C0-coiso-is-coiso}, we use Lemma \ref{lemm:charact-smooth-coiso} which is a variation of the second of the above two lemmas.

\begin{proof}[Proof of Theorem \ref{theo:special_case_smooth_Lag}]
Each of Lemmas \ref{lem:iso_charact} and \ref{lem:coiso_charact} gives a different proof.  We provide both proofs here.

\noindent \textbf{First proof: } Suppose that $H$ is any smooth Hamiltonian whose flow $\phi_H$ preserves $L'$.  Then $H\circ\theta \in C^0_{\Ham}$ and its flow, $\theta^{-1} \phi^t_H \theta$, preserves $L_0$.  It follows from Theorem \ref{theo:special_case_Lag_uniq} that the restriction of $H\circ \theta$ to $L_0$ is a function of time only.  Therefore, $H|_{L'}$ depends on time only, and so using Lemma \ref{lem:iso_charact} we conclude that  $L'$ is isotropic.

\noindent \textbf{Second proof: } Suppose that $H$ is any smooth Hamiltonian whose restriction to $L'$ is a function of time only.  Then $H\circ\theta \in C^0_{\Ham}$ and its restriction to $L_0$ depends only on time.  It follows from Theorem \ref{theo:special_case_Lag_uniq} that the flow of $H\circ \theta$, which is  $\theta^{-1} \phi^t_H \theta$, preserves $L_0$ and so the flow of $H$ preserves $L'$.  Using Lemma \ref{lem:coiso_charact} we conclude that  $L'$ is coisotropic.
\end{proof}

\begin{proof}[Proof of Theorem \ref{theo:special_case_Lag_uniq}]

\medskip
To prove the direct implication suppose that $H_t|_{L_0} = c(t)$, where $c(t)$ is a function of time only. 
For a contradiction assume that $\phi_H$ does not preserve $L_0$, then for some $t_0$ we have $\phi^{t_0}_H(L_0) \not \subset L_0$, and after the time reparametrization $t \mapsto t_0t$ we may assume that $t_0=1$, that is, $\phi_H^1(L_0)\not\subset L_0$.

Since $H \in C^0_{\Ham}$ there exists a sequence of smooth Hamiltonians $H_i \co [0,1] \times M \to \mathbb{R}$  such that $H_i$ converges uniformly to $H$ and $\phi_{H_i}$ converges to $\phi_H$ in $C^0$--topology. 

Because $\phi^1_H(L_0) \not \subset L_0$, there exists a ball $B$ such that $B \cap L_0 \neq \emptyset$ and $\phi^1_H(B) \cap L_0 = \emptyset$.  It follows that $\phi^{-1}_{H_i}(L_0) \cap B = \emptyset$ for large $i$. And so,
 $$\gamma(\phi_{H_i}^{1}(L_0), L_0)=\gamma(\phi_{H_i}^{-1}(L_0),L_0)  \geq \chzrel(B;L_0) >0.$$
Inequality \eqref{eq:gamma-bdd-osc-on-L} from Section \ref{sec:lagr-spectr-invar} implies that
$$\max_{t\in[0,1]} (\osc(H_i(t,\cdot) |_{L_0}))\geq \chzrel(B;L_0),$$ contradicting the fact that $H|_{L_0}$ is a function of time.  We conclude that $\phi_H$ preserves $L_0$.

Next, to prove the converse implication suppose that $\phi_H$ preserves $L_0$.
We will show that $H(0,\cdot)|_{L_0}$ is constant.  A time reparametrization argument, where $H(t,x)$ is replaced with $\tilde{H}(t,x) = (1-s) H(s + (1-s)t, x)$,  would then show that $H(s, \cdot)|_{L_0}$ is constant for any choice of $s \in [0,1)$.  This in turn would imply that $H|_{L_0}$ is a function of time.

Let $B$ denote an open ball intersecting $L_0$ and $U$ a small open neighborhood of $\overline B$, such that $L_0\setminus \pi(U)$ has a non-empty interior, where $\pi \co T^*L\to L_0$ is the natural projection. (Picking $U$ in this way enables us to apply Lemma \ref{lemm:haus_cont_gamma}.) Let $\psi$ be any symplectomorphism supported in $B$ and preserving $L_0$. Next, we pick $\varepsilon > 0$ such that $\phi^t_H(B), (\phi^t_H)^{-1} (B) \Subset U$ for all $t \in [0,\varepsilon]$. By a reparametrization in time, where $H(t,x)$ is replaced with $\varepsilon H(\varepsilon t, x)$, we may assume that
 $\phi^t_H(B)$, $(\phi^t_H)^{-1} (B) \Subset U$ for all $t \in [0,1]$.

 Consider the $C^0$--Hamiltonian $G = (H \circ \psi - H)\circ \phi_H$. We will now show that $G|_{L_0}=0$.   The support of $G$ is included in $\cup_{t \in [0,1]} (\phi_H^t)^{-1}(B) \subset U$, and moreover, its flow is $(\phi_H^t)^{-1}\psi^{-1}\phi_H^t\psi$.  Because $\psi$ and $\phi_H$ preserve $L_0$ the flow of $G$ also preserves $L_0$.

Since $G \in C^0_{\Ham}$ there exist smooth Hamiltonians $G_i$  such that $\{G_i\}$ converges uniformly to $G$ and $\{\phi_{G_i}\}$ converges to $\phi_G$.  Furthermore, we can require that all $G_i$'s are supported in $U$.  This can be achieved by picking a corresponding sequence of smooth Hamiltonians $H_i$ for $H$ and defining $G_i = (H_i \circ \psi - H_i)\circ \phi_{H_i}$.  For large $i$, $G_i$ is supported in $U$. 

Fix a small $r>0$.  Because $\phi^t_G(L_0) = L_0$ for any $t \in [0,1],$ for sufficiently large $i$ we have $\phi^t_{G_i}(L_0) \subset T^*_rL_0$.
Furthermore,  the Hamiltonians $G_i$ are all supported in $U$ and hence we can apply Lemma \ref{lemm:haus_cont_gamma} and conclude that $\gamma(\phi^1_{G_i}(L_0), L_0) \leq C r $, i.e $\gamma(\phi^1_{G_i}(L_0), L_0)  \to 0$.  Of course, the same reasoning yields $\gamma(\phi^t_{G_i}(L_0), L_0)  \to 0$ for all $t\in [0,1].$  Then, Proposition \ref{coro:spectral-conv-implies-constant} implies that $G|_{L_0} = c(t)$. Since it has support in $U$, we conclude that $G|_{L_0}=0$. 

In particular, $G|_{L_0}=0$ at time 0. Now since the ball $B$ can contain any chosen pair of points $x_1$, $x_2\in L_0$, and $\psi$ can be chosen so that $\psi(x_1)=x_2$, we conclude that the restriction $H(0,\cdot)|_{L_0}$ is constant.  
\end{proof}

\bibliographystyle{abbrv}
\bibliography{biblio}

\end{document}